\documentclass[oneside, a4paper,reqno]{amsart}
\usepackage{amscd,amssymb,amsmath,graphicx,verbatim,mathrsfs,xcolor}
\usepackage{wasysym}
\usepackage{hyperref}
\usepackage[TS1,OT1,T1]{fontenc}
\usepackage{lscape} 
\usepackage{fullpage} 
\usepackage[all]{xy}

\newtheorem{theorem}{Theorem}[section]
\newtheorem{lemma}[theorem]{Lemma}
\newtheorem{corollary}[theorem]{Corollary}
\newtheorem{proposition}[theorem]{Proposition}

\theoremstyle{definition}
\newtheorem{definition}[theorem]{Definition}

\newtheorem{example}[theorem]{Example}

\newtheorem{question}[theorem]{Question}

\theoremstyle{remark}
\newtheorem{remark}[theorem]{Remark}

\DeclareMathOperator{\Img}{Im} \DeclareMathOperator{\Hom}{Hom}
\DeclareMathOperator{\Ext}{Ext}

\newcommand{\Dcal}{\ensuremath{\mathcal{D}}}
\newcommand{\Xcal}{\ensuremath{\mathcal{X}}}
\newcommand{\Ycal}{\ensuremath{\mathcal{Y}}}
\newcommand{\Tcal}{\ensuremath{\mathcal{T}}}
\newcommand{\Fcal}{\ensuremath{\mathcal{F}}}
\newcommand{\Ccal}{\ensuremath{\mathcal{C}}}

\newcommand{\Acal}{\ensuremath{\mathcal{A}}}

\newcommand{\Bcal}{\ensuremath{\mathcal{B}}}
\newcommand{\Rcal}{\ensuremath{\mathcal{R}}}
\newcommand{\Hcal}{\ensuremath{\mathcal{H}}}
\newcommand{\Scal}{\ensuremath{\mathcal{S}}}

\newcommand{\Ucal}{\ensuremath{\mathcal{U}}}
\newcommand{\Vcal}{\ensuremath{\mathcal{V}}}
\newcommand{\Wcal}{\ensuremath{\mathcal{W}}}

\DeclareMathOperator{\Ker}{Ker} 
\DeclareMathOperator{\Coker}{Coker}

\numberwithin{equation}{section}

\begin{document}
\title{Properties of abelian categories via recollements}
\author{Carlos E. Parra, Jorge Vit\'oria}
\address{Carlos E. Parra, Instituto de Ciencias Fisicas y Matem\'aticas.
Edificio Emilio Pugin, cuarto piso, Campus Isla Teja, Universidad Austral de Chile, Valdivia, Chile.}
\email{carlos.parra@uach.cl}
\address{Jorge Vit\'oria, Department of Mathematics, City, University of London, Northampton Square, London EC1V 0HB, United Kingdom}
\email{jorge.vitoria@city.ac.uk}
\thanks{The first named author was partially supported by CONICYT/FONDECYT/Iniciaci\'on/11160078. The second named author was initially supported by a Marie Curie Intra-European Fellowship within the 7th European Community Framework Programme (PIEF-GA-2012-327376) at the University of Verona and, in the later stages of this project, by the Engineering and Physical Sciences Research Council of the United Kingdom, grant number EP/N016505/1 at City, University of London.}
\maketitle
\begin{abstract}
A recollement is a decomposition of a given category (abelian or triangulated) into two subcategories with functorial data that enables the glueing of structural information. This paper is dedicated to investigating the behaviour under glueing of some basic properties of abelian categories (well-poweredness, Grothendieck's axioms AB3, AB4 and AB5, existence of a generator) in the presence of a recollement. In particular, we observe that in a recollement of a Grothendieck abelian category the other two categories involved are also Grothendieck abelian and, more significantly, we provide an example where the converse does not hold and explore multiple sufficient conditions for it to hold.
\end{abstract}
\section{Introduction}
A common strategy used to study a given category is to decompose it into smaller, better understood subcategories, and to use functorial data on this decomposition to glue this understanding back to the given category. Recollements are particularly useful decompositions of abelian or triangulated categories, as they are associated with certain torsion pairs and categorical localisations. The functors in a recollement allow a transfer of properties from the given category to the subcategories involved and, crucially, vice-versa. In this paper we test this premise in the context of recollements of abelian categories, investigating elementary properties such as the exactness of certain limits and colimits or the existence of generators.

Informally, a recollement is a short exact sequence of abelian or triangulated categories such that the functors involved in this sequence admit both left and right adjoints. Recollements first appeared in the setting of triangulated categories, and they were used in the construction of the category of perverse sheaves on a singular space (\cite{BBD}). This was done by glueing t-structures, which are special kinds of torsion pairs in triangulated categories that produce homologically well-behaved abelian subcategories called hearts. It was indeed observed in \cite{BBD} that given a recollement of triangulated categories and t-structures related by glueing, the corresponding hearts relate to each other in a similar way to the underlying triangulated categories. Indeed, the hearts are related by a recollement of abelian categories. Recollements of abelian categories were later studied in detail in \cite{FP}. 

These tools are often used in representation theory to relate module categories of finite dimensional algebras (\cite{PS}), or to study polynomial functors (\cite{Kuhn}) (which arise also in algebraic topology and algebraic K-theory).  In \cite{PV}, recollements of categories of modules (i.e. where the three categories involved are module categories over rings) were essentially classified: up to Morita equivalence, they are induced by idempotent elements in a ring. Following Grothendieck's hierarchy of abelian categories, it is only natural to study recollements for the more general class of abelian categories that have exact direct limits and a generator, so called Grothendieck (abelian) categories. Such abelian categories appear naturally in algebra and algebraic geometry as categories of modules over a ring or quasi-coherent sheaves over a scheme. Our paper is a step towards understanding recollements of Grothendieck categories. Previous work in a similar direction was developed in \cite{JNN}, where the authors consider, for example, how local finiteness of Grothendieck categories (\cite[Lemma 3.2]{JNN}) or finite dimensionality of Hom-sets between objects of finite length (\cite[Lemma 4.3]{JNN}) glue along short exact sequences of abelian categories.  We summarise our main results as follows. 
\medskip

\noindent\textbf{Theorem.} {\it Let $\Acal$, $\Xcal$ and $\Ycal$ be abelian categories and consider a recollement of the form
\begin{equation}\nonumber\Rcal\colon\ \ \ \ \   \xymatrix@C=0.5cm{
\Ycal \ar[rrr]^{i_*} &&& \Acal \ar[rrr]^{j^*}  \ar @/_1.5pc/[lll]_{i^*}  \ar
 @/^1.5pc/[lll]_{i^!} &&& \Xcal
\ar @/_1.5pc/[lll]_{j_!} \ar
 @/^1.5pc/[lll]_{j_*}
 } 
 \end{equation}
If $\Acal$ admits (set-indexed) coproducts, then the following statements hold.
\begin{enumerate}
\item {\normalfont (Proposition \ref{well-p Prop})} $\Acal$ is well-powered if and only if $\Xcal$ and $\Ycal$ are well-powered;
\item {\normalfont (Theorem \ref{AB5})} $\Acal$ is AB5 if and only if $\Xcal$ and $\Ycal$ are AB5 and the recollement $\Rcal$ is directed;
\item Assume that the recollement $\Rcal$ satisfies one of the following assumptions:
\begin{enumerate}
\item {\normalfont (Corollary \ref{corollary exact Grothendieck})} $i^!$ commutes with direct limits;
\item {\normalfont (Proposition \ref{generator i! exact})} $\Acal$ admits (set-indexed) products, $i^!$ is exact and $\Rcal$ is a directed recollement;
\item {\normalfont (Proposition \ref{generator i*exact})} $i^*$ is exact and there is a generator $G_\Xcal$ of $\Xcal$ such that $\textsf{Im}(j_*)\subseteq \textsf{Gen}(j_*G_\Xcal)$;
\item {\normalfont (Corollary \ref{rec Grothendieck hearts})} $\Rcal$ is a recollement of hearts induced from glueing t-structures in a recollement of well generated triangulated categories and $\Rcal$ is a directed recollement.
\item {\normalfont (Theorem \ref{thm comp gen})} $\Rcal$ is a recollement of hearts, induced from glueing nondegenerate t-structures which fit into cosuspended TTF triples, in a recollement of compactly generated triangulated categories which restricts to a recollement between the subcategories of compact objects.
\end{enumerate}
Then $\Acal$ is Grothendieck if and only if $\Xcal$ and $\Ycal$ are Grothendieck.
\end{enumerate}}
\medskip

\noindent \textbf{Structure of the paper.} We begin with a preliminary section in which we fix notation and recall relevant definitions and results. This includes  a review of recollements, torsion pairs in abelian and triangulated categories and their glueing along recollements. In Section 3, we investigate the glueing of the property of being well-powered and of the property of having exact (co)limits (which includes Grothendieck's axioms AB4 and AB5, and their duals AB4* and AB5*). In Section 4, we look at the existence of a generator and how this property glues along a recollement. In Section 5, we focus on recollements of hearts (in well generated triangulated categories) where the problems faced in Section 4 are somehow mitigated. This section includes an example of a recollement of an abelian category $\Acal$ which is not Grothendieck by two subcategories that are module categories. We finish section 5 with a particular emphasis on recollements of hearts in compactly generated triangulated categories.

\medskip

\noindent\textbf{Acknowledgements.}
This work was partly developed in two research visits: a first one of the first named author to the University of Verona in 2014, and a second one of the second named author to the Austral University of Chile in 2017. The authors would like to thank the Department of Computer Science at the University of Verona and the Mathematics and Physics Institute at the Austral University of Chile for their hospitality and support. Moreover, both authors would like to particularly thank Lidia Angeleri H\"ugel, Francesca Mantese and Manuel Saor\'in for discussions and support in this project.

\section{Preliminaries}
In this section we recall some definitions and statements, and we fix some notation that we will use throughout the paper. We shall assume throughout that all abelian categories considered satisfy the property for any pair of objects $X$ and $Y$ the Yoneda extension group $\Ext^1(X,Y)$ is a set. Note that abelian categories arising as hearts of t-structures naturally satisfy our running assumption concerning Yoneda extensions (see \cite[Remarque 3.1.17(ii)]{BBD}).

\subsection{Limits and Colimits}
Recall that a category $\Lambda$ is \textbf{small} when the isomorphism classes of its objects form a set. If $\Ccal$ is an arbitrary category, then a functor $\Lambda \longrightarrow \Ccal$ is called a \textbf{$\Lambda$-diagram} in $\Ccal$, or simply a diagram in $\Ccal$ when $\Lambda$ is understood. The category of $\Lambda$-diagrams in $\Ccal$ will be denoted by $[\Lambda,\Ccal]$. 
We denote a $\Lambda$-diagram $X$ by $(X_\lambda)_{\lambda\in \Lambda}$, where $X_\lambda:=X(\lambda)$ for each object $\lambda$ in $\Lambda$, whenever the images by $X$ of the morphisms in $\Lambda$ are understood.

If each $\Lambda$-diagram has a colimit, we say that $\Ccal$ \textbf{admits $\Lambda$-colimits}. In this case, the functor $\textsf{colim}_\Lambda: [\Lambda,\Ccal] \longrightarrow \Ccal$ associating each $\Lambda$-diagram to its colimit is the left adjoint to the constant diagram functor $\kappa: \Ccal \longrightarrow [\Lambda,\Ccal]$.  In the case that $\Ccal$ is an abelian category, since the functor $\textsf{colim}_\Lambda$ is a left adjoint, it is right exact. We say that an abelian category $\Ccal$ has \textbf{exact $\Lambda$-colimits} if it admits $\Lambda$-colimits and the $\Lambda$-colimit functor is exact. Dually, we may also consider categories that have $\Lambda$-limits, and consider the functor $\textsf{lim}_\Lambda: [\Lambda,\Ccal] \longrightarrow \Ccal$ with dual properties.

We will be particularly interested in (co)limits when $\Lambda$ is a directed set (i.e., a preordered set $(\Lambda,\leq)$ such that every finite subset has an upper bound), viewed as a small category on which there is a unique morphisms $\lambda \longrightarrow \mu$ exactly when $\lambda\leq \mu$. Such category is said to be \textbf{directed}. The corresponding colimit functor is the \textbf{$\Lambda$-direct limit functor}  $\varinjlim_\Lambda: [\Lambda,\Ccal] \longrightarrow \Ccal$. In fact, the $\Lambda$-diagrams on $\Ccal$ are usually called \textbf{$\Lambda$-directed systems}. Dually, one has \textbf{$\Lambda$-inverse systems} and the \textbf{$\Lambda$-inverse limit functor} $\varprojlim_\Lambda: [\Lambda,\Ccal] \longrightarrow \Ccal$, whenever $\Lambda$ is a \textbf{codirected} category (i.e. whenever $\Lambda^{op}$ is directed).

An example of direct limit is the \textbf{$\Lambda$-coproduct} functor $\coprod_\Lambda: [\Lambda,\Ccal] \longrightarrow \Ccal$, when $\Lambda$ is a \textbf{discrete} category (i.e. a small category where the identities are the only morphisms). Dually, the \textbf{$\Lambda$-product} functor is also regarded as a $\Lambda$-limit when $\Lambda$ is discrete.

\subsection{Properties of abelian categories}
Abelian categories can behave quite badly. However, most of the abelian categories appearing in algebra or geometry satisfy some important properties that we recall now. Most of these properties go back to the work of Grothendieck in \cite{G}. A standard textbook reference for them is \cite[Chapters IV and V]{St}. If $\Acal$ is an abelian category, we say that $\Acal$ is
\begin{itemize}
\item \textbf{well-powered} if for each object of $\Acal$ there is only a set of subobjects (see Definition \ref{well-p});
\item \textbf{AB3} (respectively, \textbf{AB3*}) if all set-indexed coproducts (respectively, products) exist in $\Acal$;
\item \textbf{AB4} (respectively, \textbf{AB4*}) if it is AB3 (respectively, AB3*) and set-indexed coproducts (respectively, products) are exact;
\item \textbf{AB5} (respectively, \textbf{AB5*}) if it is AB3 (respectively, AB3*) and direct (respectively, inverse) limits over any directed (respectively, codirected) category are exact;
\item \textbf{Grothendieck} if it is AB5 and it has a \textbf{generator}, i.e. an object $G$ in $\Acal$ such that, for any morphism $0\neq \phi:X\longrightarrow Y$ in $\Acal$, $\Hom_\Acal(G,\phi)\neq 0$.
\end{itemize}
It is well-known that in an AB3 abelian category $\Acal$, an object $G$ is a generator if and only every object $X$ is isomorphic to a quotient of a set-indexed coproduct of copies of $G$ (i.e., there is an epimorphism $G^{(I)}\longrightarrow X$ for some set $I$). Note also that in an AB3 (respectively, AB3*) abelian category, there are $\Lambda$-colimits (respectively, $\Lambda$-limits) for all small categories $\Lambda$.

\subsection{Torsion pairs in abelian categories}\label{tp abelian} A pair of full subcategories $(\Tcal,\Fcal)$ of an abelian category $\Acal$ is said to be a \textbf{torsion pair} if 
\begin{enumerate}
\item $\Hom_\Acal(\Tcal,\Fcal)=0$.
\item For any $A$ in $\Acal$, there are $T$ in $\Tcal$, $F$ in $\Fcal$ and a short exact sequence
$$\xymatrix{0\ar[r]&T\ar[r]&A\ar[r]&F\ar[r]&0.}$$
\end{enumerate}
We then say that $\Tcal$ is a \textbf{torsion class} and $\Fcal$ is a \textbf{torsionfree class}. In the sequence above $T$ and $F$ depend functorially on $A$, and the respective functors are called the \textbf{torsion radical} and the \textbf{torsion coradical} functors. If there are two torsion pairs of the form $(\Ccal,\Tcal)$ and $(\Tcal,\Fcal)$, then $\Tcal$ is both a torsion and a torsionfree class and we call it a \textbf{TTF class}. The triple $(\Ccal,\Tcal,\Fcal)$ is then called a \textbf{TTF triple}.

If $\Acal$ is well-powered, AB3 and AB3*, it follows from \cite[Theorem 2.3]{Dickson} that a subcategory $\Tcal$ is a torsion class if and only if it is closed under epimorphic images, extensions and coproducts. Similarly, a subcategory $\Fcal$ is a torsionfree class if and only if it is closed under subobjects, extensions and products. For further information on torsion pairs in abelian categories, we refer the reader to \cite[Chapter VI]{St}.

\subsection{Torsion pairs in triangulated categories} Torsion pairs can also be defined for triangulated categories. A pair of full subcategories $(\Ucal,\Vcal)$ of a triangulated category $\Dcal$ is a \textbf{torsion pair} if 
\begin{enumerate}
\item Both $\Ucal$ and $\Vcal$ are closed under direct summands.
\item $\Hom_\Dcal(\Ucal,\Vcal)=0$.
\item For any $X$ in $\Dcal$, there are $U$ in $\Ucal$, $V$ in $\Vcal$ and a triangle
$$\xymatrix{U\ar[r]&X\ar[r]&V\ar[r]&U[1].}$$
\end{enumerate}
A torsion pair $(\Ucal,\Vcal)$ is said to be a \textbf{t-structure} if $\Ucal[1]\subseteq \Ucal$ (see \cite{BBD}), a \textbf{co-t-structure} if $\Ucal[-1]\subseteq \Ucal$ (see \cite{Bondarko, Pauk}), and \textbf{nondegenerate} if 
$$\bigcap\limits_{n\in\mathbb{Z}}\Ucal[n]=0=\bigcap\limits_{n\in\mathbb{Z}}\Vcal[n].$$
If $(\Ucal,\Vcal)$ is a t-structure, then its heart, defined as $\Ucal\cap \Vcal[1]$, is an abelian category and there is a cohomological functor $H^0:\Dcal\longrightarrow \Hcal$ (see \cite{BBD}). If $\Dcal$ admits coproducts, given a t-structure $(\Ucal,\Vcal)$, its heart is AB3 (\cite[Proposition 3.2]{PaSa}) and, if $\Vcal$ is closed under coproducts, we say that $(\Ucal,\Vcal)$ is \textbf{smashing}. In that case, the heart is even AB4 and $H^0$ commutes with coproducts (\cite[Proposition 3.3]{PaSa}).
 
A triple $(\Ucal,\Vcal,\Wcal)$ is said to be a \textbf{TTF triple} (and $\Vcal$ is called a \textbf{TTF class}) if both $(\Ucal,\Vcal)$ and $(\Vcal,\Wcal)$ are torsion pairs. Note that for such a triple, if $\Vcal[-1]\subseteq \Vcal$ then $(\Ucal,\Vcal)$ is a t-structure and $(\Vcal,\Wcal)$ is a co-t-structure, in which case we say that the triple is a \textbf{cosuspended TTF triple}.

\subsection{Recollements of abelian and triangulated categories}\label{subs rec}
A \textbf{recollement} of an abelian (respectively, triangulated) category $\Acal$ by abelian (respectively, triangulated) categories $\Bcal$ and $\Ccal$ is a diagram of additive functors as follows, satisfying the conditions below.
\begin{equation}\label{rec}\Rcal\colon\ \ \ \ \   \xymatrix@C=0.5cm{
\Ycal \ar[rrr]^{i_*} &&& \Acal \ar[rrr]^{j^*}  \ar @/_1.5pc/[lll]_{i^*}  \ar
 @/^1.5pc/[lll]_{i^!} &&& \Xcal
\ar @/_1.5pc/[lll]_{j_!} \ar
 @/^1.5pc/[lll]_{j_*}
 } 
\end{equation}

\begin{enumerate}
\item $(j_!,j^*,j_*)$ and $(i^*,i_*,i^!)$ are adjoint triples;

\item The functors $i_*$, $j_!$, and $j_*$ are fully faithful;

\item $\Img(i_*)=\Ker(j^*)$. 
\end{enumerate}

Recall that the units and counits of the adjunctions in a recollement of abelian (respectively, triangulated) categories give rise, for any object $M$ of $\Acal$, to canonical exact sequences (respectively, triangles). In the case of a recollement of abelian categories, we have exact sequences of the form
$$\xymatrix{0\ar[r]& i_*i^!M\ar[r]^{\omega_M} &M \ar[r]^{\mu_M}&j_*j^*M}\ \ \ \ \ \ \ \ \ \xymatrix{j_!j^*M\ar[r]^{\nu_M}& M\ar[r]^{\gamma_M}& i_*i^*M\ar[r]& 0}$$
where the cokernel of the last map of $\mu_M$ and the kernel of $\nu_M$ are objects lying in $\Img(i_*)$. In the case of a recollement of triangulated categories, we have triangles of the form
$$\xymatrix{ i_*i^!M\ar[r] &M \ar[r]&j_*j^*M\ar[r]&i_*i^!M[1]}\ \ \ \ \ \ \ \ \ \xymatrix{j_!j^*M\ar[r]& M\ar[r]& i_*i^*M\ar[r]&j_!j^*M[1]}$$
For further details on recollements of triangulated (respectively, abelian) categories we refer to \cite{BBD} (respectively, \cite{FP}). For a detailed study of the properties of recollements of both abelian and triangulated categories, we refer to \cite{Ps}.

Recall also that every recollement of abelian categories gives rise to a TTF triple in $\Acal$, whose notation we will keep throughout: 
$$(\Ccal,\Tcal,\Fcal):=(\Ker(i^*),\Img(i_*),\Ker(i^!)).$$
The torsion radical of $(\Ccal,\Tcal)$ is $c:=\Img(\nu)$ (with torsion coradical given by $i_*i^*$) and the torsion radical of $(\Tcal,\Fcal)$ is $i_*i^!$ (with torsion coradical given by $f:=\Img(\mu)$). See \cite[Theorem 4.3]{PV} for details. There is a similar correspondence between recollements and TTF triples in the triangulated setting (see \cite{BR}), but we will not make explicit use of it in this paper.

The correspondence between recollements of abelian categories and TTF triples allows to transfer certain properties of the functors in a recollement to those of the associated TTF triple. We would like to highlight the following easy properties.

\begin{lemma}\label{exact TTF}
Let $\Rcal$ be a recollement of abelian categories as in (\ref{rec}). Then we have that
\begin{enumerate}
\item $i^*$ is exact if and only if $\Ccal\subseteq \Fcal$;
\item $i^!$ is exact if and only if $\Fcal\subseteq \Ccal$.
\end{enumerate}
\end{lemma}
\begin{proof}
We prove statement (1). Statement (2) can be shown analogously. Suppose that $i^*$ is exact. Then class $\Ccal=\Ker(i^*)$ is closed under both quotients and subobjects. Hence, the torsion decomposition of an object $C$ in $\Ccal$ relative to the pair $(\Tcal,\Fcal)$ must lie entirely in $\Ccal$, showing that $t(C)=0$ and $C$ lies in $\Fcal$. 

Conversely, suppose that $\Ccal\subseteq \Fcal$. 
Now, given any monomorphism $g:X\longrightarrow Y$, the snake lemma guarantees a monomorphism $\Ker(i_*i^*g)\longrightarrow \Coker(c(g))$. The object $\Coker(c(g))$ clearly lies in $\Ccal$ and, therefore, it also lies in $\Fcal$. Since $\Ker(i_*i^*g)$ lies in $\Tcal$ we conclude that $\Ker(i_*i^*g)=0$ and, therefore, $i^*g$ is a monomorphism, as wanted.
\end{proof}

\subsection{Recollements of hearts}\label{glue}
Consider a recollement of triangulated categories of the form 
$$\xymatrix@C=0.5cm{
\Ycal \ar[rrr]^{i_*} &&& \Dcal \ar[rrr]^{j^*}  \ar @/_1.5pc/[lll]_{i^*}  \ar
 @/^1.5pc/[lll]_{i^!} &&& \Xcal
\ar @/_1.5pc/[lll]_{j_!} \ar
 @/^1.5pc/[lll]_{j_*}
 } $$
Suppose that $(\Ucal_\Ycal, \Vcal_\Ycal)$ and $(\Ucal_\Xcal,\Vcal_\Xcal)$ are torsion pairs in $\Ycal$ and $\Xcal$, respectively. Then, following \cite{BBD}, there is a torsion pair $(\Ucal_\Dcal,\Vcal_\Dcal)$ in $\Dcal$ defined by
\[
\Ucal_\Dcal=\big\{Z\in \Dcal \ | \ j^*(Z)\in \Ucal_\Xcal,\ i^*(Z)\in \Ucal_\Ycal \big\}, \ \ \
\Vcal_\Dcal=\big\{Z\in \Dcal \ | \ j^*(Z)\in \Vcal_\Xcal, \ i^!(X)\in \Vcal_\Ycal \big\}.
\]
It is easy to observe that the torsion pair $(\Ucal_\Dcal,\Vcal_\Dcal)$ is a t-structure (respectively, a co-t-structure) if and only if the torsion pairs $(\Ucal_\Xcal,\Vcal_\Xcal)$ and $(\Ucal_\Ycal,\Vcal_\Ycal)$ are also t-structures (respectively, co-t-structures).

In the t-structure case, the recollement of triangulated categories above induces a recollement of abelian categories relating the hearts of the given t-structures and the heart of the glued t-structure (we suggestively denote them by $\Hcal_\Ycal$, $\Hcal_\Dcal$ and $\Hcal_\Xcal$). We refer to \cite{BBD} for details. Indeed, consider the cohomological functors $H^0_{\Ycal}\colon \Ycal\longrightarrow \mathcal{H}_\Ycal$, $H^0_{\Xcal}\colon \Xcal\longrightarrow \mathcal{H}_\Xcal$ and $H^0_{\Dcal}\colon \Dcal\longrightarrow \mathcal{H}_\Dcal$, and the full embeddings $\epsilon_\Ycal$, $\epsilon_\Xcal$ and $\epsilon_\Dcal$ of the hearts into the corresponding triangulated categories. Then the induced recollement of hearts can be written as follows.
\[
\xymatrix@C=0.5cm{
\Hcal_\Ycal \ar[rrr]^{I_*} &&& \Hcal_\Dcal \ar[rrr]^{J^*}  \ar @/_1.5pc/[lll]_{I^*}  \ar
 @/^1.5pc/[lll]_{I^!} &&& \Hcal_\Xcal
\ar @/_1.5pc/[lll]_{J_!} \ar
 @/^1.5pc/[lll]_{J_*}
 }\ \ \ \ \ 
{\begin{array}{rlrl}
I^*\hspace{-0.25cm}&=H^0_\Ycal \circ i^*\circ \epsilon_\Dcal       \hspace{1.0cm}  & J_!\hspace{-0.25cm}&=H^0_\Dcal\circ j_!\circ\epsilon_\Xcal\\\noalign{\vspace{0.2cm}}
I_*\hspace{-0.25cm}&=H^0_\Dcal\circ i_*\circ \epsilon_\Ycal  \hspace{1.0cm}   & J^*\hspace{-0.25cm}&=H^0_\Xcal\circ j^*\circ\epsilon_\Dcal \\\noalign{\vspace{0.2cm}}
I^!\hspace{-0.25cm}&=H^0_\Ycal \circ i^!\circ \epsilon_\Dcal  \hspace{1.0cm}   & J_*\hspace{-0.25cm}&=H^0_\Dcal\circ j_*\circ\epsilon_\Xcal.
\end{array}}
\]
In other words, we have the following diagram:
\begin{equation}\nonumber
\xymatrix@C=0.5cm{ \Ycal \ar[ddd]_{H^0_\Ycal} \ar[rrr]^{i_*} &&& \Dcal \ar[ddd]_{H^0_\Dcal} \ar[rrr]^{j^*}  \ar @/_1.5pc/[lll]_{i^*}  \ar @/^1.5pc/[lll]^{i^!} &&& \Xcal \ar[ddd]_{H^0_\Xcal}
\ar @/_1.5pc/[lll]_{j_!} \ar
 @/^1.5pc/[lll]^{j_*} \\ 
 &&& &&& \\
 &&& &&& \\
\Hcal_\Ycal \ar @/_1.5pc/[uuu]_{\epsilon_{\Ycal}} \ar[rrr]^{I_*} &&& \Hcal_\Dcal \ar @/_1.5pc/[uuu]_{\epsilon_{\Dcal}} \ar[rrr]^{J^*}  \ar @/_1.5pc/[lll]_{I^*}  \ar
 @/^1.5pc/[lll]^{I^!} &&& \mathcal{H}_\Xcal \ar @/_1.5pc/[uuu]_{\epsilon_{\Xcal}}
\ar @/_1.5pc/[lll]_{J_!} \ar
 @/^1.5pc/[lll]^{J_*}  } 
\end{equation}
where the functors in the lower recollement are defined using the vertical functors as described above. A recollement of abelian categories obtained in this way will be called \textbf{a recollement of hearts}. 

Given an arbitrary recollement of abelian categories, we may then ask when does it arise as a recollement of hearts. The following propositions provide some sufficient conditions for this to happen. Recall that the derived category $D(\Acal)$ of an abelian category $\Acal$ (when it exists, i.e. when morphisms between any two given objects form a set) is said to be \textbf{left-complete} if, for all $X$ in $D(\Acal)$, the Milnor limit $\mathsf{Mlim}\tau^{\geq n}X$ (of the truncations associated with the standard t-structure in $D(\Acal)$) coincides with the object $X$. As shown by Neeman (\cite{Neeman}), this is a non-trivial condition. 

\begin{proposition}\label{rec hearts suf}
Let $\Rcal$ be a recollement of abelian categories as in (\ref{rec}). Assume that 
\begin{itemize}
\item $\Xcal$ is Grothendieck; 
\item $\Acal$ has an injective cogenerator;
\item $D(\Acal)$ exists, it admits products and coproducts, and it is left-complete;
\item the derived functor $j^*:D(\Acal)\longrightarrow D(\Xcal)$ commutes with products and coproducts;
\end{itemize}
Then $\Rcal$ is a recollement of hearts of smashing t-structures. 
\end{proposition}
\begin{proof}
It follows from \cite[Lemma 5.9]{Krause2} that there is a \textit{short exact sequence} of triangulated categories
$$\xymatrix{D_\Ycal(\Acal)\ar[r]^{\iota}&D(\Acal)\ar[r]^{j^*}&D(\Xcal)}$$
where $j^*$ denotes the derived functor of the corresponding (exact) functor in the recollement $\Rcal$ and $\iota$ denotes the inclusion of the full subcategory $D_\Ycal(\Acal):=\{X\in D(\Acal): H^n(X)\in i_*(\Ycal),\forall n\in\mathbb{Z}\}$ into $D(\Acal)$. By \textit{short exact sequence} we mean that $j^*$ induces a triangle equivalence between $D(\Acal)/D_\Ycal(\Acal)$ and $D(\Xcal)$. To get a recollement of triangulated categories out of this sequence we need only guarantee the existence of left and right adjoints to $j^*$ and to $\iota$. Since $\Xcal$ is Grothendieck, it is well-known (see, for example, \cite[Lemma 10]{Murfet}) that $j^*$ has a right adjoint. Also, since $\Acal$ has an injective cogenerator, $D(\Acal)$ is left-complete and $j^*$ commutes with products, it follows from a combination of \cite[Theorem 1.1]{Modoi} and \cite[Theorem 8.4.4]{Neemantria} that $j^*$ has a left adjoint. The existence of adjoints for $\iota$ then follows from the dual statement of \cite[Theorem 2.1(i)]{CPS}. Finally, note that the standard t-structure in $D(\Acal)$ is obtained by glueing t-structures in $D(\Ycal)$ and $D(\Xcal)$, since $j^*$ is t-exact with respect to the standard t-structures in $D(\Acal)$ and $D(\Ccal)$ (see \cite[Proposition 1.4.12]{BBD}). Moreover, the corresponding recollement of hearts is obtained precisely from the functor $j^*:\Acal\longrightarrow \Xcal$ and, thus, it must coincide with $\Rcal$. Finally, it follows from \cite[Corollary 4.12]{PV2} that, since $\Acal$ has an injective cogenerator, $\Acal$ is AB4. The standard t-structure in $D(\Acal)$ is therefore smashing and, thus, the cohomology functor commutes with coproducts.
\end{proof}

Slightly different sufficient conditions of a recollement of abelian categories to be a recollement of hearts can be found in \cite[Theorem 6.10]{PV2}.

\begin{proposition}\label{rec hearts suf2}
Let $\Rcal$ be a recollement of abelian categories as in (\ref{rec}). Assume that 
\begin{itemize}
\item $\Ycal$ is Grothendieck and $D(\Ycal)$ is left-complete;
\item $D(\Xcal)$ and $D(\Acal)$ exist and admit products and coproducts;
\item $D(\Acal)$ is left-complete;
\item the functor $i_*:\Ycal\longrightarrow \Acal$ induces isomorphisms $\mathsf{Ext}^n_\Ycal(M,N)\cong\mathsf{Ext}^n_\Acal(i_*M,i_*N)$ for all $M$, $N$ in $\Ycal$ and $n\geq 0$, and its derived functor $i_*:D(\Ycal)\longrightarrow D(\Acal)$  commutes with products and coproducts.
\end{itemize}
Then $\Rcal$ is a recollement of hearts of t-structures. If $\Acal$ is AB4, then these t-structures are smashing.
\end{proposition}

\section{Glueing AB-properties of abelian categories}

Recall our running assumption that we assume that all our abelian categories have the property that the Yoneda Ext-group between any two objects forms a set.

\subsection{Well-powered}
One of the most basic properties one may ask from an abelian category is that every object has only a set of subobjects (rather than a class). More precisely, for any object $M$ in an abelian category, consider the category $\mathsf{Sub}(M)$, whose objects are subobjects of $M$ and morphisms are the morphisms in $\Acal$ between such subobjects making the obvious triangle of inclusions commute. 

\begin{definition}\label{well-p}
We say that an abelian category $\Acal$ is \textbf{well-powered} if for any $M$ in $\Acal$, the category $\mathsf{Sub}(M)$ is small, i.e. $\mathsf{Sub}(M)$ has only a set of isomorphism classes.
\end{definition}

This property is useful, for example, in order to describe of torsion and torsionfree classes in an abelian category in terms of closure properties, as proved in \cite{Dickson} (see Subsection \ref{tp abelian}). We begin by proving that this property behaves well under glueing.

\begin{proposition}\label{well-p Prop}
Let $\Rcal$ be a recollement of abelian categories as in (\ref{rec}). Then $\Acal$ is well-powered if and only if $\Xcal$ and $\Ycal$ are well-powered.
\end{proposition}
\begin{proof}
Suppose that $\Acal$ is well-powered. Since $i_*$ and $j_*$ are fully faithful left exact functors, it follows immediately that both $\Xcal$ and $\Ycal$ are well-powered.

Conversely, suppose that $\Xcal$ and $\Ycal$ are well-powered. We first show that, for any $F$ in $\Fcal$, $\mathsf{Sub}(F)$ is small. Let $\alpha\colon F^\prime\longrightarrow F$ be a monomorphism and consider the following induced commutative diagram.

$$\xymatrix{0 \ar[r] & c(F^{\prime}) \ar[r] \ar@{=}[d] & F^{\prime} \ar[d]^\alpha \ar[r] & i_*i^*F^\prime \ar[r] \ar[d]^\beta & 0\\ 0 \ar[r] & c(F^{\prime}) \ar[r] & F \ar[r] & F/c(F^\prime) \ar[r] & 0}$$

First observe that $c(F^\prime)$ is a subobject of $c(F)$ and that they both lie in $\Ccal\cap \Fcal$. Since $j^*$ induces an (exact) equivalence between $\Ccal\cap \Fcal$ and $\Xcal$ (\cite{Ge}, \cite[Lemma I.1.3]{BR}) and since $\Xcal$ is well-powered, it follows that $\textsf{Sub}(c(F))\cap(\Ccal\cap\Fcal)$ is small. Hence, the objects $c(F^\prime)$ and, thus, $F/c(F^\prime)$ vary in a set. Now, observe that $i_*$ induces an equivalence of categories between $\textsf{Sub}(A)\cap \Tcal$ and $\textsf{Sub}(i_*i^!A)$, for any $A$ in $\Acal$. Therefore, since $\Ycal$ is well-powered, $\textsf{Sub}(F/c(F^\prime))\cap \Tcal$ is small and thus, since $\beta$ is a monomorphism (by the snake lemma), $i_*i^*F^\prime$ varies in a set. Since, by assumption on $\Acal$, Yoneda extensions between two objects form a set and since each $F^\prime$ represents an element in $\Ext^1_\Acal(i_*i^*F^\prime,c(F^\prime))$, we conclude that $\textsf{Sub}(F)$ is small, as wanted. With some abuse of notation we can more precisely say that
$$\textsf{Sub}(F)\subseteq \bigcup\limits_{X\in \textsf{Sub}(c(F))\cap (\Ccal\cap\Fcal)}\ \bigcup\limits_{Y\in \textsf{Sub}(F/c(F^\prime))\cap \Tcal} \Ext^1_{\Acal}(Y,X).$$

Now given an object $A$ in $\Acal$, we will see that any subobject $A^\prime$ of $A$ can be obtained as an extension of a subobject of $f(A)$ with a subobject of $i_*i^!A$. Since $\textsf{Sub}(f(A))$ and $\textsf{Sub}(i_*i^!A)$ are small, then so will be $\mathsf{Sub}(A)$, by an argument analogous to the one of the previous paragraph, thus completing the proof. Indeed, considering the inclusion map $\iota:A^\prime\longrightarrow A$, we can see that $K:=\Ker(f(\iota))$ (which is naturally an object in $\Fcal$) is also a subobject of $\Coker(i_*i^!(\iota))$. Hence $K$ lies in $\Fcal\cap\Tcal=0$, and $f(\iota)$ is a monomorphism, thus proving the claim.
\end{proof}

\subsection{Coproducts and products: AB3 and AB3*}
Recall the following well-known lemma.
\begin{lemma}\label{adj AB3}

Let $\Bcal$ be a full additive subcategory of an additive category $\Acal$, with inclusion functor $F:\Bcal\longrightarrow \Acal$, and let $\Lambda$ be a small category. Then:
\begin{enumerate}
\item If $F$ has a left adjoint and $\Acal$ admits $\Lambda$-colimits, then so does $\Bcal$.
\item If $F$ has a right adjoint and $\Acal$ admits $\Lambda$-limits, then so does $\Bcal$.
\end{enumerate}
In particular, If $\Acal$ is an AB3 (respectively, AB3*) abelian category, then so does $\Bcal$.
\end{lemma}
\begin{proof}
Assume that $F$ has a left adjoint $G$. Let us consider $M:\Lambda \longrightarrow \Bcal$ an $\Lambda$-diagram in $\Bcal$. It is easy to check that $G(\textsf{colim}_{\Lambda}(F(M_\lambda)))$ is the $\Lambda$-colimit of $M:\Lambda \longrightarrow \Bcal$ in $\Bcal$. Part (2) follows analogously.
\end{proof}

\begin{corollary}\label{AB3}
Let $\Rcal$ be a recollement of abelian categories as in (\ref{rec}). If $\Acal$ is AB3 (respectively, AB3*), then so are the categories $\Ycal$ and $\Xcal$.
\end{corollary}
\begin{proof}
If $\Acal$ is AB3, we apply Lemma \ref{adj AB3} to the adjoint pairs $(i^*,i_*)$ and $(j^*,j_*)$, obtaining that both $\Ycal$ and $\Xcal$ are AB3. Dually, we apply the second statement of the same Lemma to the adjoint pairs $(i_*,i^!)$ and $(j_!,j^*)$ to prove that if $\Acal$ is AB3*, then so are $\Ycal$ and $\Xcal$.
\end{proof}

At this point, we cannot prove or provide a counterexample to the converse assertions of the above Corollary. We will, however, either assume the existence of products and coproducts or consider contexts in which this existence comes essentially for free: when the abelian categories are hearts in triangulated categories with products and coproducts (see \cite[Proposition 3.2]{PaSa}).

\subsection{Exactness conditions: AB4, AB4*, AB5 and AB5*}
Note that the condition AB4 (respectively, AB4*) is a special circumstance of AB5 (respectively, AB5*), as a coproduct (respectively, product) is an instance of a direct limit (respectively, inverse limit) over a discrete category. If $\Lambda$ is a small category, we first prove that if $\Acal$ has  exact $\Lambda$-colimits (respectively $\Lambda$-limits), then so do $\Xcal$ and $\Ycal$.

\begin{proposition}\label{restriction AB5}
Let $\Rcal$ be a recollement of abelian categories as in (\ref{rec}), and let $\Lambda$ be a small category. If $\Acal$ has exact $\Lambda$-colimits (respectively, $\Lambda$-limits), then $\Ycal$ and $\Xcal$ have exact $\Lambda$-colimits (respectively, $\Lambda$-limits). 
In particular, if $\Acal$ is AB4 (respectively, AB4*, AB5, AB5*) then so are $\Ycal$ and $\Xcal$.
\end{proposition}
\begin{proof}
We prove the statement regarding $\Lambda$-colimits; the statement regarding $\Lambda$-limits can be shown dually. From Lemma \ref{adj AB3} $\Lambda$-colimits exist in $\Ycal$ and $\Xcal$. Since colimits are right exact, we only need to check that the $\Lambda$-colimit of a $\Lambda$-diagram of monomorphisms $g_\lambda: Z_\lambda \longrightarrow Z^{\prime}_\lambda$ in $\Xcal$ or in $\Ycal$ is still a monomorphism. Suppose first that $Z_\lambda$ lies in $\Xcal$, for all $\lambda$ in $\Lambda$. Then, we have that $j_*(g_\lambda)\colon j_*Z_\lambda\longrightarrow j_*Z^{\prime}_\lambda $ is a $\Lambda$-diagram of monomorphisms in $\Acal$, and since $\Acal$ has exact $\Lambda$-colimits, we get a monomorphism $\textsf{colim}_\Lambda(j_*(g_\lambda))\colon\textsf{colim}_{\Lambda}j_*Z_\lambda \longrightarrow \textsf{colim}_{\Lambda}j_*Z^{\prime}_\lambda$ in $\Acal$. Furthermore, since $j^{*}j_*$ is naturally equivalent to the identity functor and $j^{*}$ is an exact functor commuting with $\Lambda$-colimits, it follows that $j^{*}(\textsf{colim}_\Lambda(j_*(g_\lambda)))=\textsf{colim}_{\Lambda}(g_\lambda)$ and $\textsf{colim}_\Lambda(g_\lambda)$ is a monomorphism in $\Xcal$. 
Similarly, if  $Z_\lambda$ lies in $\Ycal$ for all $\lambda$ in $\Lambda$, since $i_{*}$ preserves $\Lambda$-colimits, we have that $\textsf{colim}_\Lambda i_*(g_\lambda)=i_*(\textsf{colim}_\Lambda g_\lambda)$ is a monomorphism. Using the fact that $i^!$ is left exact and that the composition $i^!i_*$  is naturally equivalent to the identity functor, we get that $\textsf{colim}_\Lambda g_\lambda = i^!(i_*(\textsf{colim}_\Lambda g_\lambda))$ is a monomorphism, as wanted.

Finally, observe that the condition AB4 says that $\Lambda$-colimits are exact for every discrete category $\Lambda$, while AB5 says that such $\Lambda$-colimits are exact for all directed categories $\Lambda$. Clearly if $\Acal$ satisfies either of those conditions, then so do $\Xcal$ and $\Ycal$.
\end{proof}

We will see in Subsection \ref{example} that the converse of the above statement for the AB5 condition does not hold in general. In the following we provide necessary and sufficient conditions on the recollement so that the converse implication holds. For this purpose, we consider the following definition.

\begin{definition}
Let $(\Vcal,\Wcal)$ be a torsion pair in an abelian category $\Acal$, $v$ its associated torsion radical, $w$ its torsion coradical and $\eta:v\longrightarrow 1_\Acal$ and $\mu: 1_\Acal\longrightarrow w$ the corresponding natural transformations. For a small category $\Lambda$, we say that the torsion pair $(\Vcal,\Wcal)$ is: 
\begin{itemize}
\item \textbf{$\textsf{colim}_\Lambda$-exact} (respectively, \textbf{$\textsf{lim}_\Lambda$-exact})  if, for any $\Lambda$-diagram $(X_\lambda)_{\lambda\in \Lambda}$ in $\Acal$, the morphism $\textsf{colim}_\Lambda (\eta_{X_\lambda})$  is a monomorphism (respectively, $\textsf{lim}_\Lambda(\mu_{X_\lambda})$ is an epimorphism);

\item \textbf{directed} (respectively, \textbf{codirected}) if it is $\mathsf{colim}_{\Lambda}$-exact (respectively, $\mathsf{lim}_{\Lambda^{op}}$-exact) for all directed categories $\Lambda$.
\end{itemize}
If $\Rcal$ is a recollement of abelian categories as in (\ref{rec}), then we say that $\Rcal$ is a \textbf{$\textsf{colim}_\Lambda$-exact} (respectively, \textbf{directed}, \textbf{$\textsf{lim}_\Lambda$-exact} or \textbf{codirected}) if both torsion pairs $(\Ccal,\Tcal)$ and $(\Tcal,\Fcal)$ associated to $\Rcal$ are $\textsf{colim}_\Lambda$-exact (respectively, directed, $\textsf{lim}_\Lambda$-exact or codirected)
\end{definition}

Note that in an AB5 abelian category, any torsion pair is directed. The following lemma clarifies the above definition in many contexts of interest to us.

\begin{lemma}\label{directed}
Let $\Lambda$ be a small category and let $(\Vcal,\Wcal)$ be a torsion pair in an abelian category $\Acal$, with torsion radical $v$ and torsion coradical $w$. If $\Acal$ admits $\Lambda$-colimits, then the torsion radical $v$ commutes with $\Lambda$-colimits if and only if $(\Vcal,\Wcal)$ is $\textsf{colim}_\Lambda$-exact and $\Wcal$ is closed for $\Lambda$-colimits. Dually, if $\Acal$ admits $\Lambda$-limits, then the torsion coradical $w$ commutes with $\Lambda$-limits if and only if $(\Vcal,\Wcal)$ is $\textsf{lim}_\Lambda$-exact and $\Vcal$ is closed for $\Lambda$-limits.
\end{lemma}
\begin{proof}
We prove the first statement; the second follows dually. Let $v$ and $w$ denote the torsion radical and coradical, respectively, associated to the torsion pair $(\Vcal,\Wcal)$, and let $\eta: v\longrightarrow 1_\Acal$ be the (monomorphic) natural transformation between the torsion radical and the identity functor. Let $(X_\lambda)_{\lambda\in\Lambda}$ be a $\Lambda$-diagram.

If the torsion radical $v$ commutes with $\Lambda$-colimits, then the following diagram commutes.
$$\xymatrix{  & \textsf{colim}_\Lambda(v(X_\lambda)) \ar[d]^{\cong} \ar[rr]^{\textsf{colim}_\Lambda(\eta_{X_\lambda})}& & \textsf{colim}_\Lambda(X_\lambda) \ar[r] \ar@{=}[d] & \textsf{colim}_\Lambda(w(X_\lambda)) \ar[r] \ar[d] & 0  \\ 0 \ar[r] & v(\textsf{colim}_\Lambda(X_\lambda)) \ar[rr] & & \textsf{colim}_\Lambda(X_\lambda) \ar[r] & w(\textsf{colim}_\Lambda(X_\lambda)) \ar[r] & 0}$$
This shows that $\textsf{colim}_\Lambda(\eta_{X_\lambda})$ is a monomorphism and, thus, that $(\Vcal,\Wcal)$ is $\textsf{colim}_{\Lambda}$-exact torsion pair and that $\textsf{colim}_\Lambda w(X_\lambda)\cong w(\textsf{colim}_\Lambda(X_\lambda))$ lies in $\Wcal$.

Conversely, consider the following $\Lambda$-diagram of short exact sequences in $\Acal$.
$$\xymatrix{0\ar[r]& v(X_\lambda)\ar[r]^{\eta_{X_\lambda}}&X_\lambda\ar[r] & w(X_\lambda)\ar[r]&0}$$
If $(\Vcal,\Wcal)$ is $\textsf{colim}_\Lambda$-exact and $\Wcal$ is closed for $\Lambda$-colimits, then the $\Lambda$-colimit of the diagram of short exact sequences is short exact, with the first term lying in $\Vcal$ (since torsion classes are closed for $\Lambda$-colimits) and the last term lying in $\Wcal$. This finishes our proof.
\end{proof}

Recall that, given a recollement as in (\ref{rec}), the natural transformation $\omega\colon i_*i^!\longrightarrow 1_\Acal$ coincides with the one associated to the torsion radical of $(\Tcal,\Fcal)$. Let us also denote by $h\colon c\longrightarrow 1_\Acal$ the natural transformation associated to the torsion radical of the pair $(\Ccal,\Tcal)$.

\begin{example}\label{example directed} Let $\Rcal$ be a recollement as in (\ref{rec}) and $\Lambda$ a small category such that $\Acal$ admits $\Lambda$-colimits and $\Ycal$ has exact $\Lambda$-colimits.
\begin{enumerate}
\item If $i^!$ commutes with $\Lambda$-colimits, then $\Rcal$ is a $\mathsf{colim}_\Lambda$-exact recollement. Indeed, since the torsion radical of $(\Tcal,\Fcal)$ is $i_*i^!$ and $\Fcal=\Ker(i^!)$ (hence closed under $\Lambda$-colimits), Lemma \ref{directed} shows that $(\Tcal,\Fcal)$ is $\mathsf{colim}_\Lambda$-exact. Given a $\Lambda$-diagram $(M_\lambda)_{\lambda\in\Lambda}$ in $\Acal$, consider the commutative $\Lambda$-diagram of exact sequences.
$$\xymatrix{0\ar[r]&i_*i^!c(M_\lambda)\ar[r]\ar[d]^{i_*i^!(h_\lambda)}& c(M_\lambda)\ar[r]\ar[d]^{h_\lambda}&f(c(M_\lambda))\ar[r]\ar[d]^{f(h_\lambda)}&0\\ 0\ar[r]& i_*i^!M_\lambda \ar[r]&M_\lambda\ar[r]& f(M_\lambda)\ar[r]&0}$$
Since $(\Tcal,\Fcal)$ is $\mathsf{colim}_\Lambda$-exact, both rows remain exact when applying $\mathsf{colim}_\Lambda$ and, since $\Ycal$ has exact $\Lambda$-colimits, $\mathsf{colim}_\Lambda(i_*i^!(h_\lambda))$ is a monomorphism. By the snake lemma, there is a monomorphism from $\Ker(\mathsf{colim}_\Lambda(h_\lambda))$ to $\Ker(\mathsf{colim}_\Lambda(f(h_\lambda)))$. Since $\Fcal$ is closed under $\Lambda$-colimits and subobjects we conclude that $\Ker(\mathsf{colim}_\Lambda(h_\lambda))$ lies in $\Fcal$. On the other hand, since $j^*c(M_\lambda)\cong M_\lambda$, applying $j^*$ to $\mathsf{colim}_\Lambda(h_\lambda)$ allows us to observe that $\Ker(\mathsf{colim}_\Lambda(h_\lambda))$ also lies in $\Tcal$ and, thus, it must be zero, proving our claim.

\item If $i^*$ is exact, then $\Rcal$ is a $\mathsf{colim}_\Lambda$-exact recollement. Recall from Lemma \ref{exact TTF} that $i^*$ is exact if and only if $\Ccal\subseteq \Fcal$. Given a $\Lambda$-diagram $(M_\lambda)_{\lambda\in\Lambda}$ in $\Acal$, it follows (as in (1) above) that the kernel of the $\Lambda$-colimit of the torsion radical maps $c(M_\lambda)\longrightarrow M_\lambda$ lies in in $\Tcal$. Since $\Ccal$ is closed under $\Lambda$-colimits (it is a torsion class) and $\Ccal\subseteq \Fcal$, we conclude that that kernel is zero, and $(\Ccal,\Tcal)$ is $\mathsf{colim}_\Lambda$-exact. On the other hand, consider the commutative diagram of canonical maps.
$$\xymatrix{0\ar[r]&0\ar[r]\ar[d]& i_*i^!M_\lambda\ar@{=}[r]\ar[d]^{\omega_\lambda}&i_*i^!M_\lambda\ar[r]\ar[d]^{\gamma_\lambda \omega_\lambda}&0\\ 0\ar[r]& c(M_\lambda) \ar[r]^{h_\lambda}&M_\lambda\ar[r]^{\gamma_\lambda}& i_*i^*M_\lambda\ar[r]&0}$$
Note that the kernel of $\gamma_\lambda \omega_\lambda$ lies in $\Tcal$ and is a subobject ($\omega_\lambda$ is a monomorphism) of $c(M_\lambda)$, thus lying also in $\Fcal$. In other words, $\gamma_\lambda \omega_\lambda$ is a monomorphism. When applying to the diagram $\mathsf{colim}_\Lambda$, the rows remain exact (since $(\Ccal,\Tcal)$ is $\mathsf{colim}_\Lambda$-exact) and $\mathsf{colim}_\Lambda(\gamma_\lambda \omega_\lambda)$ remains a monomorphism since $\Ycal$ has exact $\Lambda$-colimits. Hence $\mathsf{colim}(\omega_\lambda)$ is a monomorphism and $(\Tcal,\Fcal)$ is $\mathsf{colim}_\Lambda$-exact.
\end{enumerate}
\end{example}

The following theorem states that in a recollement as before, the exactness of $\Lambda$-colimits (respectively $\Lambda$-limits) in $\Acal$  depend on the exactness on $\Xcal$ and $\Ycal$ and on the $\textsf{colim}_{\Lambda}$-exactness (respectively $\textsf{lim}_{\Lambda}$-exactness) of the recollement.

\begin{theorem}\label{AB5}
Let $\Rcal$ be a recollement of abelian categories as in (\ref{rec}). 
\begin{enumerate}
\item Suppose $\Acal$ is AB3. Then $\Lambda$-colimits in $\Acal$ are exact if and only if they are exact in $\Xcal$ and $\Ycal$ and $\Rcal$ is $\textsf{colim}_{\Lambda}$-exact. In particular, $\Acal$ is AB5 if and only if $\Xcal$ and $\Ycal$ are AB5 and $\Rcal$ is directed.
\item Suppose $\Acal$ is AB3*. Then $\Lambda$-limits in $\Acal$ are exact if and only if they are exact in $\Xcal$ and $\Ycal$ and $\Rcal$ is \textsf{lim}$_{\Lambda}$-exact.  In particular, $\Acal$ is AB5* if and only if $\Xcal$ and $\Ycal$ are AB5* and $\Rcal$ is codirected.
\end{enumerate}
\end{theorem}

\begin{proof}
Once again, we prove one of the assertions; the other one is dual. Let $\Lambda$ be a small category. If $\Acal$ has exact $\Lambda$-colimits, it follows from Proposition \ref{restriction AB5} that both $\Xcal$ and $\Ycal$ also satisfy this property. It is also clear that any torsion pair in such abelian category is $\textsf{colim}_\Lambda$-exact. 

Conversely, suppose that $\Lambda$-colimits in $\Xcal$ and $\Ycal$ are exact and that both torsion pairs in the TTF triple $(\Ccal,\Tcal,\Fcal)$ are $\textsf{colim}_\Lambda$-exact. Let $(g_\lambda:M_\lambda\longrightarrow N_\lambda)_{\lambda\in\Lambda}$ be a $\Lambda$-diagram of monomorphisms in $\Acal$. First we prove the result in two special instances.

\textbf{Claim 1:} $\textsf{colim}_{\Lambda}(g_\lambda)$ is a monomorphism whenever $M_\lambda$ lies in $\Tcal$ for all $\lambda$ in $\Lambda$. Indeed, if $(M_{\lambda})_{\lambda\in\Lambda}$ is a $\Lambda$-diagram in $\Tcal$. Then $g_\lambda$ factors through $\omega_{N_\lambda}:i_*i^!N_\lambda \longrightarrow N_\lambda$, i.e., there is $\bar{g}_\lambda: M_\lambda\longrightarrow i_*i^!N_\lambda$ such that $g_\lambda=\omega_{N_\lambda}\bar{g}_\lambda$, for all $\lambda\in\Lambda$. Note that all these maps are monomorphisms. Now, since $\Lambda$-colimits are exact in $\Ycal$ and $\bar{g}_\lambda$ is a morphism in $\Tcal=i_*\Ycal$ for all $\lambda\in\Lambda$, it follows that $\textsf{colim}_{\Lambda}(\bar{g}_\lambda)$ is a monomorphism. Moreover, since $(\Tcal,\Fcal)$ is a $\textsf{colim}_{\Lambda}$-exact torsion pair, then $\textsf{colim}_{\Lambda} (\omega_{N_\lambda})$ is also a monomorphism. Hence, $\textsf{colim}_{\Lambda}(g_\lambda) = \textsf{colim}_{\Lambda}(\omega_{N_\lambda}) \circ \textsf{colim}_{\Lambda}(\bar{g}_\lambda)$ is a monomorphism.

\textbf{Claim 2:} $\textsf{colim}_{\Lambda}(g_\lambda)$ is a monomorphism whenever both $M_\lambda$ and $N_\lambda$ lie in $\Ccal$ for all $\lambda$ in $\Lambda$. Indeed, if $(g_\lambda)_{\lambda\in\Lambda}$ is a $\Lambda$-diagram of monomorphisms in $\Ccal$ for each $\lambda\in\Lambda$, there is a commutative $\Lambda$-diagram
$$\xymatrix{ 0 \ar[r] & \Ker(\nu_{M_{\lambda}}) \ar[r] \ar[d]^{k_\lambda} & j_{!}j^{*}M_\lambda \ar[r]^{\nu_{M_\lambda}} \ar[d]_{j_!j^*(g_\lambda)}& M_\lambda \ar[r] \ar[d]^{g_\lambda} & 0 \\0 \ar[r] & \Ker(\nu_{N_\lambda}) \ar[r]  & j_{!}j^{*}N_\lambda \ar[r]^{\nu_{N_\lambda}} & N_\lambda \ar[r]  & 0 }$$
where both $\Ker(\nu_{M_{\lambda}})$ and $\Ker(\nu_{N_{\lambda}})$ can be seen to lie in $\Tcal$ (just apply the exact functor $j^*$ to the exact sequences above). Since $g_\lambda$ is a monomorphism, we have that $\Ker(k_\lambda)\cong \Ker(j_!j^*(g_\lambda))$ and they both lie in $\Tcal$ since they are (isomorphic to) a subobject of $\Ker(\nu_{M_{\lambda}})$. Now consider the $\Lambda$-colimit of the diagram. By Claim 1, the rows of the diagram so obtained remain exact and, again by Claim 1, we have 
$$\Ker(\textsf{colim}_\Lambda(k_\lambda))=\textsf{colim}_{\Lambda}(\Ker(k_\lambda))\cong \textsf{colim}_{\Lambda}( \Ker(j_!j^*(g_\lambda)))=\Ker(\textsf{colim}_{\Lambda}(j_!j^*(g_\lambda)))$$
and, therefore, by the snake lemma, there is a monomorphism $\Ker(\textsf{colim}_{\Lambda}(g_\lambda))\longrightarrow \Coker(\textsf{colim}_{\Lambda}(k_\lambda))=\textsf{colim}_{\Lambda}(\Coker(k_\lambda))$. Since $\Coker(k_\lambda)$ lies in $\Tcal$, and the natural map $d_\lambda: \Coker(k_\lambda)\longrightarrow \Coker(j_!j^*(g_\lambda))$ is a monomorphism, again by Claim 1 we have that $\textsf{colim}_{\Lambda}(d_\lambda)$ is a monomorphism. Therefore, the snake lemma tells us that $\Ker(\textsf{colim}_{\Lambda}(g_\lambda))=0$, proving Claim 2. 

We are now ready to show the general case. Let $(g_\lambda)_{\lambda\in\Lambda}$ be a $\Lambda$-diagram of monomorphisms as before and consider the following commutative diagrams:
$$\xymatrix{ 0 \ar[r] & c(M_{\lambda}) \ar[r] \ar[d]^{c(g_\lambda)} & M_\lambda \ar[r] \ar[d]^{g_\lambda}& i_*i^*(M_\lambda) \ar[r] \ar[d]^{i_*i^*(g_\lambda)}& 0 \\0 \ar[r] & c(N_{\lambda}) \ar[r] & N_\lambda \ar[r] & i_*i^*(N_\lambda) \ar[r] & 0 } $$
for $\lambda$ in $\Lambda$, induced by the torsion decompositions of $M_\lambda$ and $N_{\lambda}$. Since $g_\lambda$ is a monomorphism, then so is $c(g_\lambda)$. Note also that there is a $\Lambda$-diagram of monomorphisms $d_\lambda:\Ker(i_*i^*(g_\lambda))\longrightarrow \Coker(c(g_\lambda))$ arising from the snake lemma, with $\Ker(i_*i^*(g_\lambda))$ lying in $\Tcal$. Applying the $\Lambda$-colimit functor to the diagram, the rows will remain exact due to the fact that $(\Ccal,\Tcal)$ is $\textsf{colim}_{\Lambda}$-exact, $\textsf{colim}_{\Lambda}(c(g_\lambda))$ is a monomorphism by Claim 2 and $\textsf{colim}_{\Lambda}(d_\lambda)$ is a monomorphism by Claim 1. Hence, it follows from the snake lemma that $\textsf{colim}_{\Lambda}(g_\lambda)$ is a monomorphism, as wanted.
\end{proof}

\begin{remark}
Note that since $\Tcal$ is closed for direct limits, the requirement that $(\Ccal,\Tcal)$ is directed is equivalent, by Lemma \ref{directed}, to the requirement that the torsion radical $c$ commutes with direct limits.
\end{remark}

\begin{corollary}\label{corollary exact AB5}
Let $\Rcal$ be a recollement of abelian categories as in (\ref{rec}) with $\Acal$ an AB3 abelian category. If $i^!$ commutes with $\Lambda$-colimits or if $i^*$ is exact, then $\Xcal$ and $\Ycal$ have exact $\Lambda$-colimits if and only if so does $\Acal$. In particular, under either of these two assumptions
\begin{enumerate}
\item $\Xcal$ and $\Ycal$ are AB4 if and only if $\Acal$ is AB4;
\item $\Xcal$ and $\Ycal$ are AB5 if and only if $\Acal$ is AB5.
\end{enumerate}
\end{corollary}
\begin{proof}
This follows directly from Example \ref{example directed} and Theorem \ref{AB5}. 
\end{proof}

\section{Generators and Grothendieck categories}
In this section we treat the more difficult problem of building a generator for the abelian category at the centre of a recollement. We begin with an easy observation. 
\begin{lemma}\label{AllsGrothendieck}
Let $\Rcal$ be a recollement of abelian categories as in (\ref{rec}) and suppose that $\Acal$ has a generator. Then both $\Xcal$ and $\Ycal$ have a generator. Moreover, if $\Acal$ is a Grothendieck category, then so are $\Xcal$ and $\Ycal$.
\end{lemma}
\begin{proof}
Given a generator $G$ in $\Acal$, $j^*G$ and $i^*G$ are generators in $\Xcal$ and $\Ycal$ respectively, since the right adjoints $j_*$ and $i_*$ are fully faithful. The final statement follows from Proposition \ref{restriction AB5}.
\end{proof}

In full generality, the question of whether the existence of generators in $\Xcal$ and $\Ycal$ implies the existence of a generator in $\Acal$ remains open for an arbitrary recollement of abelian categories (see Question \ref{question}). We will, however, cover some cases, using various naturally occurring assumptions on the recollement. Our constructions use a fair amount of categorical/homological tricks. The following lemma exemplifies the sort of arguments that we will use. Although the lemma itself will only be used later in the paper, we prove it now so that the reader can become familiar with the technique. The key standard fact used here is that any set-indexed coproduct is a direct limit of its finite subcoproducts. We will refer to the use of this fact in the construction of generators as \textit{generating by finite reduction}. Given a subcategory $\Bcal$ of an abelian category $\Acal$, we will denote by $\mathsf{add}(\Bcal)$ the class of summands of finite direct sums of objects in $\Bcal$ and by $\mathsf{Gen}(\Bcal)$ the class of quotients of arbitrary coproducts of objects in $\Bcal$. If $\Bcal$ contains a single object $X$, we write $\mathsf{add}(X)$ and $\mathsf{Gen}(X)$ for simplicity.

\begin{lemma}\label{generating by finite reduction}
Let $\Acal$ be an AB3 abelian category. Let $X$ be an object of $\Acal$ and $Y$ a subobject of $X$. If $X$ is generated by $M$, then $Y$ is generated by subobjects of finite direct sums of $M$. In other words, we have $\mathsf{Sub}(\mathsf{Gen}(M))\subseteq \mathsf{Gen}(\mathsf{Sub}(\mathsf{add}(M)))$
\end{lemma}
\begin{proof}
Let $f:Y\longrightarrow X$ be the inclusion map and let $g:M^{(I)}\longrightarrow X$ be an epimorphism from a coproduct of $M$ to $X$. We consider the pullback diagram as follows
$$\xymatrix{0\ar[r]&N\ar[r]\ar[d]&M^{(I)}\ar[r]\ar[d]^g&C\ar[r]\ar@{=}[d]&0 \\ 0\ar[r]&Y\ar[r]^f&X\ar[r]&C\ar[r]&0}$$
from which we deduce that $Y$ is a quotient of $N$.
It is well-known that the coproduct $M^{(I)}$ can be written as the direct limit of finite coproducts of $M$ indexed by all finite subsets of $I$, i.e. $M^{(I)}\cong \varinjlim_{F\subseteq_{\text{fin}} I} M^{(F)}$. For each subset $F$, let $g_F$ denote the restriction of $g$ to $M^{(F)}$ and consider the pullback diagram 
$$\xymatrix{0\ar[r]&N_F\ar[r]\ar[d]&M^{(F)}\ar[r]\ar[d]^{g_F}&C\ar[r]\ar@{=}[d]&0 \\ 0\ar[r]&Y\ar[r]^f&X\ar[r]&C\ar[r]&0}$$
Since direct limits in an abelian category are always right exact, we get an exact sequence 
$$\varinjlim_{F\subseteq_{\text{fin}} I} N_F\longrightarrow M^{(I)}\longrightarrow C\longrightarrow 0$$
from which we deduce that $N$ is a quotient of $\varinjlim_{F\subseteq_{\text{fin}} I} N_F$ and, hence, a quotient of the coproduct of all $N_F$. This then shows that $Y$ is a quotient of that same coproduct, i.e. it is generated by  subobjects of finite direct sums of $M$, as wanted.
\end{proof}
\begin{remark}
Note that if $\Acal$ is well-powered, then the isoclasses of objects in $\mathsf{Sub}(\mathsf{add}(M))$ form a set!
\end{remark}

\subsection{Constructing generators I: commutation with direct limits}
Let us now begin with the actual construction of generators. For a general TTF triple in an abelian category, we have the following theorem, whose proof will also rely on generation by finite reduction. This argument is inspired by a similar construction in \cite[Proposition 4.7]{PaSa}. 

\begin{theorem}\label{TTF generator}
Let $\Acal$ be an AB3, well-powered abelian category, admitting a TTF triple $(\Ccal,\Tcal,\Fcal)$. If $\Fcal$ is closed under direct limits and there are objects $C$ and $T$ in $\Acal$ such that $\Ccal=\mathsf{Gen}(C)$ and $\Tcal=\mathsf{Gen}(T)$, then $\Acal$ has a generator.
\end{theorem}
\begin{proof}
Let us denote by $t$ and $c$ the torsion radicals of $(\Tcal,\Fcal)$ and $(\Ccal,\Tcal)$ respectively (similarly, $(1\colon t)$ and $(1\colon c)$ will denote the corresponding torsion coradicals). Let $M$ be an object of $\Acal$ and consider the universal morphism $u:C^{(U)}\longrightarrow M$, with $U=\Hom_\Acal(C,M)$. Since $\Ccal=\mathsf{Gen}(C)$, the image of $u$ coincides with $c(M)$ and, therefore, $\Coker(u)=(1\colon c)(M)$. For each finite subset $F$ of $U$, we consider the split map $e_F:C^{(F)}\longrightarrow C^{(U)}$ and, with $K:=\Ker(u)$, $K_F:=\Ker(u\circ e_F)$ and $X_F:=\Coker(u\circ e_F)$, we obtain the following commutative diagram, where the top right square and the bottom left square are both pullback squares.
$$\xymatrix{ 0 \ar[r] & K_F \ar[r] \ar@{=}[d] & C^{(F)} \ar[r] \ar@{=}[d] & M_F \ar[r]^{\rho_F} \ar[d]^{\omega_F}  & t(X_F) \ar@{->}[d] \ar[r] & 0 \\ 0 \ar[r] & K_F \ar[r] \ar@{->}[d]  & C^{(F)} \ar[r] \ar[d]^{e_F} & M \ar[r] \ar@{=}[d] & X_F \ar[r] \ar[d] & 0 \\ 0 \ar[r] & K \ar[r] & C^{(U)} \ar[r]^u & M \ar[r] & (1\colon c)(M) \ar[r] & 0}$$
Note that we have $\varinjlim X_F\cong (1\colon c)(M)$, which lies in $\Tcal$. Since $\varinjlim (1\colon t)(X_F)$ is a quotient of $\varinjlim X_F$, it lies in $\Tcal$. However, by assumption (3), $\varinjlim (1\colon t)(X_F)$ must also lie in $\Fcal$ and therefore it is zero. This shows that $\varinjlim t(X_F)\cong \varinjlim X_F \cong (1\colon c)(M)$. Thus, $M$ is a quotient of the coproduct of all $M_F$ where $F$ varies in the set of finite subsets of $U$. We will now build from $C$ and $T$ a set generating the objects $M_F$.

Since $\Tcal=\mathsf{Gen}(T)$, consider an epimorphism $q:T^{(U_F)} \longrightarrow t(X_F)$, where $U_F=\Hom_\Acal(T,t(X_F))$. For each finite subset $E_F \subseteq U_F$, we consider the split inclusion $s_{E_F}:T^{(E_F)}\longrightarrow T^{(U_F)}$ and we can build following commutative diagram, by two consecutive pullbacks.
$$\xymatrix{0 \ar[r] & \Ker(\rho_F) \ar[r] \ar@{=}[d] & M_{E_F} \ar[r] \ar[d] & T^{(E_F)} \ar[d] \ar[r] & 0\\ 0 \ar[r] & \Ker(\rho_F) \ar[r] \ar@{=}[d] & \overline{M_F} \ar[r] \ar[d] & T^{(U_F)} \ar[d] \ar[r] & 0\\ 0 \ar[r] & \Ker(\rho_F) \ar[r] & M_F \ar[r]^{\rho_F} & t(X_F) \ar[r] & 0}$$ 
As before, $\overline{M_F}$ is a quotient of the coproduct of all $M_{E_F}$ and therefore so is $M_F$. To complete our argument we observe that $M_{E_F}$ varies in a set, thus generating the category. Indeed, since $\Ker(\rho_F)$ is a quotient of a finite coproduct of copies of $C$, and Yoneda Ext-groups form a set in $\Acal$ by assumption, we conclude that all $M_{E_F}$ are in bijection with a subset of 
$$\underset{m,n\in \mathbb{N}}{\bigcup} \hspace{0.2 cm} \underset{V\leq C^m}{\bigcup} \hspace{0.2 cm } \Ext^{1}_{\Acal}(T^n,C^m/V)$$
thus forming a set of generators for $\Acal$, as wanted.
\end{proof}

\begin{corollary}\label{corollary exact Grothendieck}
Let $\Rcal$ be a recollement as in (\ref{rec}), and $\Acal$ an AB3 abelian category. If $i^!$ commutes with direct limits, then $\Acal$ is a Grothendieck category if and only if $\Xcal$ and $\Ycal$ are Grothendieck categories.
\end{corollary}
\begin{proof}
Suppose that $\Xcal$ and $\Ycal$ are Grothendieck. Since $i^!$ commutes with direct limits, by Corollary \ref{corollary exact AB5} $\Acal$ is AB5. We now check the assumptions of Theorem \ref{TTF generator} in order to guarantee that $\Acal$ has a generator. From Proposition \ref{well-p Prop}, since $\Xcal$ and $\Ycal$ are well-powered, then so is $\Acal$. The fact that $i^!$ commutes with direct limits guarantees that $\Fcal$ is closed under direct limits. Finally, given generators $G_\Xcal$ and $G_\Ycal$ in $\Xcal$ and $\Ycal$ respectively, since $i_*G_\Ycal$ lies in $\Tcal$, $j_!G_\Xcal$ lies in $\Ccal$ and both $\Tcal$ and $\Ccal$ are closed for coproducts and quotients, we have that $\mathsf{Gen}(j_!G_\Xcal)\subseteq\Ccal$ and $\mathsf{Gen}(i_*G_\Ycal)\subseteq\Tcal$. Since $i_*i^*T^\prime\cong T^\prime$ for any $T^\prime$ in $\Tcal$ and $C^\prime$ is a quotient of $j_!j^*C^\prime$ for any $C^\prime$ in $\Ccal$, the right exactness of $j_!$ and $i_*$ guarantee the reverse inclusions.
\end{proof}

\subsection{Constructing generators II: exact functors}
In this section we build generators under the assumption that one of $i^!$ or $i^*$ is exact. Before that, however, we begin with a simple observation.

\begin{lemma}\label{comm coprod}
Let $\Rcal$ be a recollement of abelian categories as in (\ref{rec}), with $\Acal$ satisfying AB3*, AB4 and the property that the coproduct of a set of objects in $\Acal$ is a subobject of the corresponding product. Then $i^!$ commutes with coproducts. In particular, if $\Acal$ is AB3* and AB5, then $i^!$ commutes with coproducts.
\end{lemma}
\begin{proof}
Given a family $(M_\lambda)_{\lambda\in\Lambda}$ in $\Acal$, since coproducts are exact by assumption, we get a sequence
$$0\longrightarrow \coprod_{\lambda\in\Lambda}i_*i^!M_\lambda\longrightarrow \coprod_{\lambda\in\Lambda}M_\lambda\longrightarrow \coprod_{\lambda\in\Lambda}f(M_\lambda)\longrightarrow 0.$$
Now, since coproducts are subobjects of products, $\Fcal$ is closed for coproducts. Hence, this sequence must be the torsion decomposition of $\coprod_{\lambda\in\Lambda}M_\lambda$ with respect to the torsion pair $(\Tcal,\Fcal)$. This shows that $\coprod_{\lambda\in\Lambda}i_*i^!M_\lambda=i_*i^!\coprod_{\lambda\in\Lambda}M_\lambda$, as wanted. For the final statement, recall that the canonical morphism from the coproduct of a family of objects to the product of the same family is the direct limit of the of the (split) monomorphisms from finite sub-coproducts to the product. Hence the result follows for an AB5 abelian category with products.
\end{proof}

This allows us to state the following easy proposition.

\begin{proposition}\label{generator i! exact}
Let $\Rcal$ be a recollement of abelian categories as in (\ref{rec}), with $\Acal$ an AB3 and AB3* abelian category. Suppose that $i^!$ is exact. Then $\Acal$ is Grothendieck if and only if $\Xcal$ and $\Ycal$ are Grothendieck and $\Rcal$ is directed.
\end{proposition}
\begin{proof}
If $\Xcal$ and $\Ycal$ are Grothendieck and $\Rcal$ is directed, then by Theorem $\ref{AB5}$, $\Acal$ is AB5. It also follows from Lemma \ref{comm coprod} that $i^!$ commutes with coproducts and, therefore, since it is exact by assumption, it commutes with direct limits. The result then follows from Corollary \ref{corollary exact Grothendieck}.
\end{proof}

\begin{remark}
There is an advantage in the construction of a generator when $i^!$ is exact.  If $\Acal$ is AB3 as before and if $G_\Xcal$ and $G_\Ycal$ are generators in $\Xcal$ and $\Ycal$ respectively, then it is easy to see that, since $Ext^1_\Acal(j_!G_\Xcal,\Tcal)=0$, the object $j_!G_\Xcal\oplus i_*G_\Ycal$ is a generator of $\Acal$.
\end{remark}

Let us proceed to the more delicate case of when $i^*$ is exact. We begin by investigating what kind of necessary conditions occur when $\Acal$ is a Grothendieck category (or even more generally, an AB4 abelian category with a generator). In such case, given a generator $G$ of $\Acal$ we have, in particular, $\textsf{Im}(j_{*})\subseteq \textsf{Gen}(G)$. Moreover, we also have that $\Img(j_*)\subseteq \mathsf{Gen}(j_*j^*G)$. Indeed, given $M$ in the image of $j_*$ (i.e. $M\cong j_*j^*M$), let $\alpha=\textsf{Hom}_{\Acal}(G,M)$ and denote by $u:G^{(\alpha)}\longrightarrow M$ the universal map, which is an epimorphism. This map must clearly factor through $j_*j^*(G^{(\alpha)})$. Considering the canonical morphism $\theta:(j_*j^*G)^{(\alpha)}\longrightarrow j_*j^*(G^{(\alpha)})$ and using the fact that coproducts are exact in $\Acal$, we get the following commutative diagram with exact rows.
$$\xymatrix{0 \ar[r] & (i_*i^{!}G)^{(\alpha)} \ar[r] \ar[d]  & G^{(\alpha)} \ar[r] \ar@{=}[d]  & (j_{*}j^*G)^{(\alpha)} \ar[d]^\theta \\ 0 \ar[r] & i_*i^{!}(G^{(\alpha)}) \ar[r] & G^{(\alpha)} \ar@{>>}[d]^u \ar[r] & j_{*}j^*(G^{(\alpha)}) \ar[d] \\ && M \ar[r]^\cong & j_{*}j^{*}(M)}$$
It is then easy to conclude from the commutativity of the diagram that the composition in the last column of the diagram is an epimorphism, as wanted. Note also that $j^*G$ is a generator of $\Xcal$ (see also Lemma \ref{AllsGrothendieck}). The following proposition proves that given a recollement where $\Xcal$ and $\Ycal$ are Grothendieck, then this peculiar property of having an object generating the image of $j_*$ suffices to build a generator in $\Acal$.
\begin{proposition}\label{generator i*exact}
Let $\Rcal$ be a recollement as in (\ref{rec}) such that $i^{*}$ is exact and $\Acal$ is an AB3 abelian category. Then $\Acal$ is a Grothendieck category if and only if $\Xcal$ and $\Ycal$ are Grothendieck abelian categories and there exists $G_\Xcal$ generator of $\Xcal$ such that $\textsf{Im}(j_*)\subseteq \textsf{Gen}(j_*G_\Xcal)$.
\end{proposition}
\begin{proof}
Following the previous paragraph, we only need to prove one direction. Let $\Xcal$ and $\Ycal$ be Grothendieck categories. Since $i^*$ is exact, it follows from Example \ref{example directed}(2) and Theorem \ref{AB5} that $\Acal$ is AB5. Hence, we only need to show that $\Acal$ has a generator. Under our assumption, Lemma \ref{exact TTF} shows that $\Ccal \subseteq \Fcal$ and, for each object $M$ in $\Acal$, we have an exact sequence
$$\xymatrix{0 \ar[r] & j_!j^{*}(M) \ar[r] & M \ar[r] & i_*i^{*}(M) \ar[r] & 0.}$$
Moreover, the canonical morphism $j_{!}j^{*}(M) \longrightarrow j_{*}j^{*}(M)$ is a monomorphism and we can consider the pushout of the above sequence along this map. Since $\textsf{Ext}^{1}_{\Acal}(\Tcal,\textsf{Im}(j_*))=0$, this new exact sequence will split and we obtain that $M$ is a subobject of $j_*j^{*}(M)\oplus i_*i^{*}(M)$. Hence, by Lemma \ref{generating by finite reduction}, it is enough to find an object that generates objects of the form $j_*j^{*}(M)\oplus i_*i^{*}(M)$. The result then follows since, by assumption, $j_*j^{*}(M)\oplus i_*i^{*}(M)$ lies in $\textsf{Gen}(j_*(G_\Xcal)\oplus i_*(G_\Ycal))$, for some $G_\Ycal$ generator of $\Ycal$.
\end{proof}

In all cases explored in this section, we have built generators in $\Acal$ from the images of generators of $\Xcal$ and $\Ycal$ via the fully faithful functors in the recollement, and using generation by finite reduction.
\begin{question}\label{question}
Let $\Rcal$ be a recollement as in (\ref{rec}) such that $\Acal$ is AB3, and suppose that $G_\Xcal$ and $G_\Ycal$ are generators of $\Xcal$ and $\Ycal$ respectively. Is it true that the set of subojects of finite direct sums of $j_*G_\Xcal\oplus j_!G_\Xcal\oplus i_*G_\Ycal$ generates $\Acal$?
\end{question}

\section{Recollements of hearts}\label{Section hearts}

As recalled in Section 2, one way of building examples of recollements of abelian categories is via glueing t-structures in a recollement of triangulated categories. 
Many recollements of abelian categories arise in this way, as portrayed in Propositions \ref{rec hearts suf} and \ref{rec hearts suf2}. It is often the case that little is known about the hearts obtained by glueing t-structures. The techniques and results developed in the previous sections contribute to shed some light on this problem.

\subsection{An example}\label{example}
We begin with an example of a recollement of hearts where the outer terms are Grothendieck categories (even categories of vector spaces over a field) and the middle term is not an AB5 abelian category. In particular, this provides a counterexample to the converse of Proposition \ref{restriction AB5}.

Let $A$ be the path algebra over a field $\mathbb{K}$ of a quiver with two vertices, 1 and 2,
and countably many arrows from 1 to 2 (the countable Kronecker). Equivalently, we may think of $A$ as the matrix ring 
$$A=\left(\begin{array}{cc} \mathbb{K} & 0 \\ \mathbb{K}^{(\mathbb{N})} & \mathbb{K}\end{array}\right)$$
Let $\mathsf{Mod}$-$A$ denote the category of right $A$-modules and let $e=e_1$ be the idempotent element associated with the vertex 1. It is well-known that it induces a recollement of module categories as follows: \\
$$ \xymatrix@C=0.5cm{\mathsf{Mod}\text{-}A/AeA \ar[rrr]^{f_*} &&& \mathsf{Mod}\text{-}A \ar[rrr]^{(-)e}  \ar @/_1.5pc/[lll]_{-\otimes_AA/AeA}  \ar @/^1.5pc/[lll]^{\Hom_A(A/AeA,-)} &&& \mathsf{Mod}\text{-}eAe\ar @/_1.5pc/[lll]_{-\otimes_{eAe}eA} \ar @/^1.5pc/[lll]^{\Hom_{eAe}(Ae,-)}} $$\\
where $f_*$ is the restriction of scalars functor associated to the ring epimorphism $f:A\longrightarrow A/AeA$. Note that, since $A$ is hereditary and $\mathsf{Tor}_1^A(A/AeA,A/AeA)=0$, this recollement of module categories induces naturally a recollement of their derived module categories (see \cite{GL}, and compare with Proposition \ref{rec hearts suf2}).\\ 
$$ \xymatrix@C=0.5cm{D(A/AeA) \ar[rrr]^{f_*} &&& D(A) \ar[rrr]  \ar @/_1.5pc/[lll]  \ar @/^1.5pc/[lll] &&& D(e_1Ae_1)\ar @/_1.5pc/[lll] \ar @/^1.5pc/[lll]} $$\\
Since $\mathsf{Mod}\text{-}A$ is AB5 (and thus $(\Tcal, \Fcal)$ is a directed torsion pair), if $\mathcal{F}$ were closed for direct limits, then $\Hom_{A}(A/AeA,-)$ would commute with direct limits, which is not the case since $A/AeA$ is not finitely presented. Thus, $\Fcal$ is not closed for direct limits and, therefore, by \cite[Theorem 4.8]{PaSa}, the heart 
$$\Hcal=\{X\in D(A): H^{-1}(X)\in\Fcal, H^0(X)\in\Tcal, H^k(X)=0, \forall k\neq -1,0\}$$ 
of the t-structure in $D(A)$ associated with the torsion pair $(\Tcal,\Fcal)$ (see \cite[Proposition 2.1]{HRS} for the construction of this t-structure) is not an AB5 abelian category. Using the techniques explored in detail in \cite[Section 6]{LVY}, there is a recollement of hearts of the form\\
$$ \xymatrix@C=0.5cm{\mathsf{Mod}\text{-}A/AeA \ar[rrr] &&& \Hcal \ar[rrr]  \ar @/_1.5pc/[lll]  \ar @/^1.5pc/[lll] &&& \mathsf{Mod}\text{-}eAe[1]\ar @/_1.5pc/[lll] \ar @/^1.5pc/[lll]} $$\\

\noindent where the outer terms are AB5 (even Grothendieck) categories and the middle term is not an AB5 abelian category. It is known from \cite[Proposition 3.3]{PaSa} that $\Hcal$ is AB4 and AB4*, and it also follows from Proposition \ref{well-p Prop} that it is well-powered. For more details on constructing recollements of hearts from recollements of abelian categories in the way exemplified above, we refer to \cite[Section 10]{PsSurvey}.

\subsection{Hearts in well generated triangulated categories} We aim to prove that in a recollement of hearts arising from \textit{nice enough} triangulated categories, the question of whether the property of being Grothendieck abelian glues well can be simplified. We recall the definition of the type of triangulated categories we will work with (see \cite{Neemantria,Krause}).

\begin{definition}
Given a regular cardinal $\alpha$ and $\Tcal$ a triangulated category, we say that 
\begin{itemize}
\item an object $X$ in $\Tcal$ is \textbf{$\alpha$-small} if given any map $h:X\longrightarrow \coprod_{\lambda\in\Lambda} Y_\lambda$ for some family of objects $(Y_\lambda)_{\lambda\in\Lambda}$ in $\Tcal$, the map $h$ factors through a subcoproduct $\coprod_{\omega\in \Omega}Y_\omega$ where $\Omega$ is a subset of $\Lambda$ of cardinal strictly less than $\alpha$;
\item $\Tcal$ is \textbf{$\alpha$-well generated} if it has set-indexed coproducts and it has a set of objects $\Scal$ such that
\begin{enumerate}
\item[(G1)] if $X$ is an object of $\Tcal$ such that $\Hom_\Tcal(S,X)=0$ for all $S\in\Scal$, then $X=0$;
\item[(G2)] for every set of maps $(g_\lambda:X_\lambda\longrightarrow Y_\lambda)_{\lambda\in \Lambda}$ in $\Tcal$, if $\Hom_\Tcal(S,g_\lambda)$ is surjective for all $\lambda$ in $\Lambda$ and all $S$ in $\Scal$, then $\Hom_\Tcal(S,\coprod_{\lambda\in \Lambda}g_\lambda)$ is surjective for all $S$ in $\Scal$;
\item[(G3)] every object $S$ in $\Scal$ is $\alpha$-small.
\end{enumerate}
\item $\Tcal$ is \textbf{well generated} if it is $\alpha$-well generated for some regular cardinal $\alpha$. 
\item if $\Tcal$ is $\alpha$-well generated by a set $\Scal$, then we say that an object $X$ in $\Tcal$ is \textbf{$\beta$-compact} (for $\beta\geq \alpha$) if $X$ lies in the smallest triangulated subcategory closed under coproducts with less than $\beta$ factors and containing $\Scal$; we denote the category of $\beta$-compact objects by $\Tcal^\beta$.
\end{itemize}
\end{definition}

It is known that in an $\alpha$-well generated triangulated category $\Tcal$, the subcategories $\Tcal^\beta$, for $\beta\geq \alpha$ are small and do not depend on the set of $\alpha$-small generators $\Scal$ (\cite[Lemma 5]{Krause}). Moreover we have that $\Tcal$ is the union, over all $\beta\geq \alpha$ of $\Tcal^\beta$ (\cite[Corollary]{Krause}). If $\Tcal$ is $\alpha$-well generated, we denote by $\mathsf{Mod}\text{-}\Tcal^\alpha$ the category of additive (contravariant) functors $(\Tcal^{\alpha})^{\text{op}}\longrightarrow \mathsf{Mod}\text{-}\mathbb{Z}$. This is known to be an AB4 and AB4* abelian category with exact $\Lambda$-direct limits for any directed category $\Lambda$ with cardinality $\alpha$ (and, thus, for any directed category with cardinality $\beta\geq \alpha$), and with enough projectives (the representable functors $\Hom_\Tcal(-,X)$ with $X$ in $\Tcal^\alpha$, see \cite{Neemantria,Krause0} for details). 

We are particularly interested in well generated triangulated categories because of the following result.

\begin{proposition}\label{generator heart}
Let $\Acal$ be an AB5 abelian category. Then the following are equivalent.
\begin{enumerate}
\item $\Acal$ is Grothendieck.
\item $D(\Acal)$ exists and is well generated.
\item $\Acal$ is the heart of a smashing t-structure in a well generated triangulated category.
\end{enumerate}
\end{proposition}
\begin{proof}
It is well-known that the derived category of a Grothendieck category exists (\cite{Spaltenstein}) and that it is well generated (\cite{Neeman0}), yielding that (1) implies (2). Also, considering the standard t-structure in $D(\Acal)$, we easily obtain that (2) implies (3). It remains to prove that (3) implies (1).

Let $\Dcal$ be a well generated triangulated category with a smashing t-structure of which $\Acal$ is the heart and whose asssociated cohomological functor is $H^0\colon\Dcal\longrightarrow \Acal$. Let $\alpha$ be a regular cardinal such that $\Dcal$ is $\alpha$-well generated and consider the Yoneda functor $\mathbf{y}_\alpha:\Dcal\longrightarrow \mathsf{Mod}\text{-}\Dcal^\alpha$ sending and object $X$ to the functor $\Hom_{\Dcal}(-,X)_{|\Dcal^\alpha}$. Since $\Acal$ is assumed to be AB5, we have that in particular that direct limits over directed categories with cardinality $\alpha$ are exact in $\Acal$ and since the t-structure is smashing, the functor $H^0$ commutes with coproducts. Therefore, from \cite[Proposition 6.10.1]{Krause} $H^0$ must factor through $\mathbf{y}_\alpha$, i.e. there is a unique exact functor $F:\mathsf{Mod}\text{-}\Dcal^\alpha\longrightarrow \Acal$ such that the following diagram commutes.
$$\xymatrix{\Dcal\ar[rr]^{\mathbf{y}_\alpha}\ar[rd]_{H^0}&&\mathsf{Mod}\text{-}\Dcal^\alpha\ar[ld]^F\\ & \Acal}$$
Since $F$ is exact and dense and $\mathsf{Mod}\text{-}\Dcal^\alpha$ has a set of generators (namely, $\{\mathbf{y}_\alpha(X):X\in\Dcal^\alpha\}$), then also $\Acal$ has a set of generators (namely, $\{H^0(X):X\in \Dcal^\alpha\}$).
\end{proof}

\begin{corollary}\label{rec Grothendieck hearts}
Let $\Rcal$ be a recollement of abelian categories as in (\ref{rec}). If $\Rcal$ is a recollement of hearts of smashing t-structures glued from a recollement of well generated triangulated categories, then $\Acal$ is Grothendieck if and only if $\Xcal$ and $\Ycal$ are Grothendieck and $\Rcal$ is directed.
\end{corollary}
\begin{proof}
We only need to prove that if $\Xcal$ and $\Ycal$ are Grothendieck and $\Rcal$ is directed then $\Acal$ is Grothendick. Since $\Acal$ is a heart in a triangulated category with coproducts, it follows from \cite{PaSa} that $\Acal$ is AB3 (even AB4 since the t-structure is smashing) and we may, therefore, use Theorem \ref{AB5} to conclude that $\Acal$ is AB5. Now, since $\Acal$ is the heart of a smashing t-structure in a well generated triangulated category, Proposition \ref{generator heart} shows that $\Acal$ is Grothendieck.
\end{proof}

We now combine the above Corollary with Propositions \ref{rec hearts suf} and \ref{rec hearts suf2}.

\begin{corollary}\label{rec hearts cor1}
Let $\Rcal$ be a recollement of abelian categories as in (\ref{rec}). Assume that 
\begin{itemize}
\item $\Acal$ has an injective cogenerator;
\item $D(\Xcal)$ and $D(\Acal)$ exist and admit products and coproducts;
\item the derived functor $j^*:D(\Acal)\longrightarrow D(\Xcal)$ commutes with products;
\item $D(\Acal)$ is well generated and left-complete.
\end{itemize}
Then $\Acal$ is Grothendieck if and only if $\Xcal$ and $\Ycal$ are Grothendieck and $\Rcal$ is directed.
\end{corollary}
\begin{proof}
This result follows as a combination of Proposition \ref{rec hearts suf} and Corollary \ref{rec Grothendieck hearts}. Note that since $\Acal$ is assumed to have an injective cogenerator, $\Acal$ is AB4 (see, for example, \cite[Corollary 4.12]{PV2}) and therefore the derived functor of $j^*$ commutes automatically with coproducts.
\end{proof}

\begin{corollary}\label{rec hearts cor2}
Let $\Rcal$ be a recollement of abelian categories as in (\ref{rec}). Assume that 
\begin{itemize}
\item the derived categories $D(\Ycal)$ and $D(\Acal)$ exist and admit products and coproducts;
\item the derived functor $i_*:D(\Ycal)\longrightarrow D(\Acal)$ commutes with products;
\item the functor $i_*:\Ycal\longrightarrow \Acal$ induces isomorphisms $\mathsf{Ext}^n_\Ycal(M,N)\cong\mathsf{Ext}^n_\Acal(i_*M,i_*N)$ for all $M$ and $N$ in $\Ycal$ and for all $n\geq 0$.
\item $D(\Acal)$ and $D(\Ycal)$ are left complete, and $D(\Acal)$ is well generated.
\end{itemize}
Then $\Acal$ is Grothendieck if and only if $\Xcal$ and $\Ycal$ are Grothendieck and $\Rcal$ is directed.
\end{corollary}
\begin{proof}
This result follows as a combination of Proposition \ref{rec hearts suf2} and Corollary \ref{rec Grothendieck hearts}. Indeed, note that $i_*$ will commute with coproducts given either of the assumptions in the equivalence stated above.
\end{proof}

\subsection{Hearts in compactly generated triangulated categories}
Recall that a triangulated category $\Dcal$ is said to be \textbf{compactly generated} if it is $\aleph_0$-well generated, and we denote its $\aleph_0$-compact objects (which are precisely those objects $X$ such that $\Hom_\Dcal(X,-)$ commutes with arbitrary coproducts i.e., the compact objects of $\Dcal$) by $\Dcal^c$. Examples of such categories are the derived categories of small differential graded categories (in fact, up to equivalence, these are known to be precisely the algebraic compactly generated triangulated categories, see \cite{Keller} for details).
We will use the following result.

\begin{theorem}\label{Gao-Ps}\cite[Proposition 3.4]{GPs}
Let $\Xcal$, $\Ycal$ and $\Dcal$ be compactly generated triangulated categories and suppose that there is a recollement
\begin{equation}\nonumber\Rcal\colon\ \ \ \ \   \xymatrix@C=0.5cm{
\Ycal \ar[rrr]^{i_*} &&& \Dcal \ar[rrr]^{j^*}  \ar @/_1.5pc/[lll]_{i^*}  \ar
 @/^1.5pc/[lll]_{i^!} &&& \Xcal
\ar @/_1.5pc/[lll]_{j_!} \ar
 @/^1.5pc/[lll]_{j_*}
 } 
 \end{equation}
 The following are equivalent.
 \begin{enumerate}
 \item The functor $i_*$ preserves compact objects;
 \item The functor $j^*$ preserves compact objects;
 \item The recollement $\Rcal$ restricts to a recollement between the full subcategories of compact objects;
 \item $i^!$ and $j_*$ admit right adjoints, yielding a reflected recollement.
 \end{enumerate}
 
\end{theorem}

Taking a different approach from the previous sections, we will use some techniques from \cite{AMV} to produce recollements of hearts which are recollements of Grothendieck categories. As before, we consider the category $\mathsf{Mod}\text{-}\Dcal^c$ of additive (contravariant) functors $(\Dcal^c)^{op}\longrightarrow \mathsf{Mod}\text{-}\mathbb{Z}$. 
It is well-known that $\mathsf{Mod}\text{-}\Dcal^c$ is a Grothendieck category and that $\mathbf{y}_{\aleph_0}:\Dcal\longrightarrow \mathsf{Mod}\text{-}\Dcal^c$ (in this subsection denoted simply by $\mathbf{y}$) is not, in general, fully faithful. Moreover, it is easy to see that $\mathbf{y}$ sends triangles to long exact sequences.

\begin{definition}
Let $\Dcal$ be a compactly generated triangulated category. 
\begin{enumerate}
\item A triangle in $\Dcal$ is said to be \textbf{pure} if its image under $\mathbf{y}$ is a short exact sequence. 
\item A covariant additive functor $F:\Dcal\longrightarrow \mathsf{Mod}\text{-}\mathbb{Z}$ is \textbf{coherent} if there are compact objects $K$ and $L$ and an exact sequence 
$$Hom_\Dcal(K,-)\longrightarrow Hom_\Dcal(L,-)\longrightarrow F\longrightarrow 0.$$
The category of coherent functors is denoted by $\mathsf{Coh}$-$\Dcal$. 
\item A subcategory $\Vcal$ of $\Dcal$ is said to be \textbf{definable} if there is a set of functors $(F_i)_{i\in I}$ in $\mathsf{Coh}$-$\Dcal$ such that $X$ lies in $\Vcal$ if and only if $F_i(X)=0$ for all $i$ in $I$.
\end{enumerate}
\end{definition}

We use the following criterion to determine whether certain hearts are Grothendieck categories.

\begin{theorem}\cite[Theorems 3.6 and 4.9]{AMV}\label{Grothendieck definable}
Let $\Dcal$ be a compactly generated triangulated category and $(\Ucal,\Vcal,\Wcal)$ a cosuspended TTF triple in $\Tcal$ such that $(\Ucal,\Vcal)$ is nondegenerate. The following are equivalent.
\begin{enumerate}
\item The heart $\Hcal_\mathbb{T}$ is a Grothendieck abelian category.
\item $\Vcal$ is definable.
\end{enumerate}
\end{theorem}

Given a recollement of hearts induced from glueing (a special kind of) t-structures in a recollement of compactly generated triangulated categories, we are then able to use the ideas introduced above to provide a criterion for when the middle term is a Grothendieck abelian category.

\begin{theorem}\label{thm comp gen}
Let $\Dcal, \Xcal$, and $\Ycal$ be compactly generated triangulated categories. Suppose that $\Rcal$ is a recollement of triangulated categories as in (\ref{rec}) satisfying the equivalent conditions of Theorem \ref{Gao-Ps}. Let $(\Ucal_\Xcal,\Vcal_\Xcal,\Wcal_\Xcal)$ and $(\Ucal_\Ycal,\Vcal_\Ycal,\Wcal_\Ycal)$ be cosuspended TTF triples in $\Xcal$ and $\Ycal$ respectively, such that $(\Ucal_\Xcal,\Vcal_\Xcal)$ and $(\Ucal_\Ycal,\Vcal_\Ycal)$ are nondegenerate t-structures. The following statements hold for the glued t-structure $(\Ucal_\Dcal, \Vcal_\Dcal)$ in $\Dcal$.
\begin{enumerate}
\item The t-structure $(\Ucal_\Dcal, \Vcal_\Dcal)$ is nondegenerate and there is a cosuspended TTF triple $(\Ucal_\Dcal, \Vcal_\Dcal,\Wcal_\Dcal)$
\item The hearts $\Hcal_\Xcal:=\Ucal_\Xcal[-1]\cap\Vcal_\Xcal$ and $\Hcal_\Ycal:=\Ucal_\Ycal[-1]\cap \Vcal_\Ycal$ are Grothendieck if and only if the heart $\Hcal_\Dcal:=\Ucal_\Dcal[-1]\cap\Vcal_\Dcal$ is Grothendieck.
\end{enumerate}
\end{theorem}

\begin{proof}
(1) Since $\Rcal$ admits a lower reflected recollement, i.e. the right adjoints of $i^!$ and $j_*$ give rise to a new recollement, it can be easily checked that glueing the co-t-structures $(\Vcal_\Xcal,\Wcal_\Xcal)$ and $(\Vcal_\Ycal,\Wcal_\Ycal)$ along the lower recollement yield a co-t-structure of the form $(\Vcal_\Dcal,\Wcal_\Dcal)$ (see also \cite[Remark 2.6]{LVY}).

(2) Since the t-structure $(\Ucal_\Dcal,\Vcal_\Dcal)$ is built by glueing $(\Ucal_\Xcal,\Vcal_\Xcal)$ and $(\Ucal_\Ycal,\Vcal_\Ycal)$, we know that there is a recollement of $\Hcal_\Dcal$ by $\Hcal_\Xcal$ and $\Hcal_\Ycal$, see Subsection \ref{glue}. If $\Hcal_\Dcal$ is Grothendieck, then, from Lemma \ref{AllsGrothendieck}, so are $\Hcal_\Xcal$ and $\Hcal_\Ycal$. Let us now prove the reverse direction. Using Theorem \ref{Grothendieck definable} we know that $\Vcal_\Xcal$ and $\Vcal_\Ycal$ are definable and we only need to prove that $\Vcal_\Dcal$ is definable as well. Let $(F_j)_{j\in J}$ and $(G_k)_{k\in K}$ be two families of functors in $\mathsf{Coh}$-$\Ycal$ and $\mathsf{Coh}$-$\Xcal$ respectively, such that an object $Y\in\Ycal$ lies in $\Vcal_\Ycal$ if and only if $F_j(Y)=0$ for all $j\in J$ and such that an object $X\in\Xcal$ lies in $\Vcal_\Xcal$ if and only if $G_k(X)=0$ for all $k\in K$. It is then clear, by the glueing formula for t-structures that 
$$\Vcal_\Dcal=\{Z\in\Dcal: F_ji^!(Z)=0, G_kj^*(Z)=0, \forall j\in J, \forall k\in K\}.$$
Hence, to conclude the proof, it is enough to show that $F_ji^!$ and $G_kj^*$ are coherent functors, thus showing that $\Vcal_\Dcal$ is definable and that, by Theorem \ref{Grothendieck definable}, the heart $\Hcal_\Dcal$ is Grothendieck. By \cite[Proposition 5.1]{Kr2}, a functor is coherent if and only if it preserves coproducts, products and it sends pure triangles to short exact sequences. Since $F_j$ and $G_k$ are coherent functors over $\Ycal$ and $\Xcal$ respectively, and since the functors $j^*$ and $i^!$ have both left and right adjoints (by assumption on $\Rcal$), we easily see that $F_ji^!$ and $G_kj^*$ preserve both products and coproducts. It remains to see that they send pure triangles to short exact sequences. For this, it is enough to see that both $i^!$ and $j^*$ preserve pure triangles and then we use the fact that $F_j$ and $G_k$ are coherent. Let $\Delta$ be a pure triangle in $\Dcal$. This means precisely that $\mathbf{y}\Delta$ is a short exact sequence in $Mod(\Dcal^c)$, i.e. for any compact object $C$ in $\Tcal^c$, the sequence $Hom_\Dcal(C,\Delta)$ is short exact. Now, observe that given an object $C^\prime\in\Xcal^c$, the sequence $Hom_\Xcal(C^\prime,j^*\Delta)\cong Hom_\Dcal(j_!C^\prime,\Delta)$ is short exact since $j_!C^\prime$ is compact (the functor $j_!$ preserves compact objects in any recollement). On the other hand, given $C^{\prime\prime}\in\Ycal^c$, the sequence $Hom_\Ycal(C^{\prime\prime},i_!\Delta)\cong Hom_\Dcal(i_*C^{\prime\prime},\Delta)$ is short exact since $i_*$ preserves compact objects by our assumption on $\Rcal$. This finishes the proof.
\end{proof}

\end{document}